\documentclass[10pt]{amsart}
\calclayout

\usepackage{amssymb,amsthm,mathtools,thmtools,cite,hyperref,tikz-cd}
\allowdisplaybreaks

\DeclareMathOperator{\Cl}{Cl}
\DeclareMathOperator{\id}{id}
\DeclareMathOperator{\spec}{Spec}
\DeclareMathOperator{\End}{End}
\DeclareMathOperator{\Der}{Der}
\DeclareMathOperator{\Hom}{Hom}

\DeclareMathOperator{\Lie}{Lie}
\DeclareMathOperator{\adj}{adj}
\DeclareMathOperator{\diag}{diag}
\DeclareMathOperator{\Span}{span}
\DeclareMathOperator{\pure}{pure}
\DeclareMathOperator{\Sing}{Sing}

\DeclareMathOperator{\Dist}{Dist}
\DeclareMathOperator{\GL}{GL}

\DeclareMathOperator{\cdet}{cdet}
\DeclareMathOperator{\rdet}{rdet}

\newcommand{\co}{\mathcal{O}}
\newcommand{\cx}{\mathcal{X}}
\newcommand{\cs}{\mathcal{S}}
\newcommand{\cd}{\mathcal{D}}
\newcommand{\cpinf}{\mathcal{P}_\infty}

\newcommand{\bfX}{\mathbf{X}}
\newcommand{\bfY}{\mathbf{Y}}
\newcommand{\bfH}{\mathbf{H}}
\newcommand{\bfXi}{\mathbf{\Xi}}

\newcommand{\fq}{\mathfrak{q}}
\newcommand{\fn}{\mathfrak{n}}
\newcommand{\fb}{\mathfrak{b}}
\newcommand{\fh}{\mathfrak{h}}
\newcommand{\fgl}{\mathfrak{gl}}
\newcommand{\ba}{\mathbb{A}}
\newcommand{\bc}{\mathbb{C}}
\newcommand{\bn}{\mathbb{N}}
\newcommand{\bz}{\mathbb{Z}}
\newcommand{\lra}[2]{\langle{#1},{#2}\rangle}
\newcommand{\bone}{\boldsymbol{1}}

\numberwithin{equation}{section}

\theoremstyle{plain}
\newtheorem{theorem}{Theorem}[section]
\newtheorem{proposition}[theorem]{Proposition}
\newtheorem{lemma}[theorem]{Lemma}
\newtheorem{corollary}[theorem]{Corollary}
\newtheorem*{theoremA}{Theorem A}
\newtheorem*{theoremB}{Theorem B}
\newtheorem*{theoremC}{Theorem C}

\theoremstyle{definition}
\newtheorem{definition}[theorem]{Definition}
\newtheorem{example}[theorem]{Example}

\theoremstyle{remark}
\newtheorem{remark}[theorem]{Remark}

\title[Highest weight vectors in queer Howe duality]{Highest weight vectors in queer Howe duality}

\author{Minghao Bi}
\address{School of Mathematical Sciences, University of Science and Technology of China, Hefei, Anhui, 230006, China}
\email{mhbi@mail.ustc.edu.cn}

\author{Hongjia Chen}
\address{School of Mathematical Sciences, University of Science and Technology of China, Hefei, Anhui, 230006, China}
\email{hjchen@ustc.edu.cn}

\author{Longhui Wang}
\address{School of Mathematical Sciences, University of Science and Technology of China, Hefei, Anhui, 230006, China}
\email{lhwang@ustc.edu.cn}

\author{Yongjie Wang}
\address{School of Mathematics, Hefei University of Technology, Hefei, Anhui, 230009, China}
\email{wyjie@mail.ustc.edu.cn}

\thanks{This paper is partially supported by the National Key R\&D Program of China (2024YFA1013802), the NSF of China (12401036, 12471025), the Anhui Provincial Natural Science Foundation (2308085MA01), and the Fundamental Research Funds for the Central Universities (JZ2025HGTG0251).}


\begin{document}

\begin{abstract}
We use Gauss decomposition to study the singular vectors in the regular representation on the coordinate superalgebra of the big cell of the queer supergroup. This gives explicit bases for the highest weight spaces in queer Howe duality. As an application, the same method recovers the explicit highest weight vectors in the general linear case obtained by Cheng and Wang.
\end{abstract}

\maketitle

\section{Introduction}

\subsection{}

Howe duality originates from Howe's representation-theoretic approach to classical invariant theory \cite{MR986027,MR1321638}.  It concerns pairs of mutually commuting actions of Lie groups or Lie algebras on a common algebra, where the resulting representation admits a multiplicity-free decomposition.

A basic example is the \((\fgl(m),\fgl(n))\)-Howe duality on \(S(\bc^m\otimes \bc^n)\).  Under the natural commuting actions of \(\fgl(m)\) and \(\fgl(n)\), one has the multiplicity-free decomposition
\[
S(\bc^m\otimes \bc^n)
 \cong
\bigoplus_{\substack{\lambda\in\mathcal{P}\\ \ell(\lambda)\leq \min(m,n)}}
V_m(\lambda)\otimes V_n(\lambda),
\]
where \(V_r(\lambda)\) denotes the irreducible \(\fgl(r)\)-module of highest weight \(\lambda\). Moreover, the highest weight vectors in this decomposition are products of leading minors. Their explicit forms play a role in relating, for the general linear Lie algebra, the Howe duality, the Schur--Weyl duality, and the first fundamental theorem (FFT) of invariant theory \cite{MR1321638}.

Several analogues of Howe duality for Lie superalgebras have been developed. For the general linear Lie superalgebra case, Cheng and Wang \cite{MR1847665} established the \((\fgl(p|q),\fgl(m|n))\)-Howe duality on \(S(\bc^{p|q}\otimes\bc^{m|n})\), together with explicit formulas for the corresponding highest weight vectors.  Further developments and applications appear in \cite{MR1847665,MR3012224}.

Another important analogue is the queer case. The \((\fq(m),\fq(n))\)-Howe duality was established independently by Cheng and Wang \cite{MR1786858}, and by Sergeev \cite{MR1827078}.  Unlike the general linear case, \(\bc^{m|m}\otimes\bc^{n|n}\) splits into two isomorphic irreducible \(\fq(m)\oplus\fq(n)\)-summands; we denote the irreducible representation by \(\bc^{mn|mn}\).  We have the following multiplicity-free decomposition
\[
S(\bc^{mn|mn})\cong
\bigoplus_{\substack{\lambda\in\mathcal{SP}\\ \ell(\lambda)\leq \min(m,n)}}
2^{-\delta(\lambda)}L_m(\lambda)\otimes L_n(\lambda),
\]
where \(\delta(\lambda)\) is the parity of \(\ell(\lambda)\) and \(L_r(\lambda)\) denotes the irreducible \(\fq(r)\)-module of highest weight \(\lambda\). This decomposition is the representation-theoretic form of the Cauchy identity for Schur \(Q\)-functions, and it is closely related to the Schur--Sergeev duality.

Quantum analogues of Howe duality have also been developed in various settings. Zhang \cite{MR1974323} established the quantum general linear case using quantum coordinate algebras. Chang and Wang \cite{MR4040727} later established the quantum queer case using quantum coordinate queer superalgebras as noncommutative analogues of symmetric superalgebras. In the quantum queer case, the Howe duality also recovers the Sergeev--Olshanski duality and leads to the corresponding FFT; more details can be found in \cite{MR5005254,Chang2025}.

We observe that the classical \((\fgl(n),\fgl(n))\)-Howe duality can be viewed as the subrepresentation \(\co(M(n))\subset \co(\GL(n))\) in the algebraic Peter--Weyl theorem for \(\GL(n)\). This observation suggests studying Howe duality through the Peter--Weyl theorem.
The theorem describes functions on Lie groups in terms of matrix coefficients of their irreducible representations \cite{MR1512386}. In its algebraic form, the coordinate algebra \(\co(G)\) of a reductive algebraic group carries commuting left and right regular actions, and decomposes accordingly as a \(G\times G\)-module \cite{MR2522486}.  

Styrkas \cite{styrkas2004regularrepresentationbigcell} extended this viewpoint to regular functions on the big cell of a complex simple algebraic group, obtaining a Peter--Weyl type description in terms of big projective modules in the BGG category \(\co\). Scheunert and Zhang \cite{MR1927431} studied the Peter--Weyl theory for the general linear supergroup via the Hopf superalgebra structure on its coordinate superalgebra of regular functions.

\subsection{}

The main purpose of this paper is to determine the highest weight vectors in the Howe duality for the queer Lie superalgebra. Motivated by the classical general linear case, where such vectors are given by products of leading minors, we seek analogous explicit generators for the corresponding highest weight vectors in the \((\fq(m),\fq(n))\)-Howe dualities.

Compared with the basic Lie superalgebras, the queer Lie superalgebra has several distinctive features. For instance, its Cartan subalgebra \(\fh\) is non-abelian, and thus the highest weight space of \(L_n(\lambda)\) carries a natural Clifford module structure \cite[Section~1.5.4]{MR3012224}.

For a strict partition \(\lambda\) of length \(\ell\), the highest weight vector spaces associated with the \(\lambda\)-summand in the \((\fq(m),\fq(n))\)-Howe dualities are naturally identified whenever they exist. We denote the common space by \(H_\lambda\) and view it as a subspace of
\[
\cpinf:=\bc[x_{ij},y_{ij}\mid i,j\in\bn].
\]
Thus, it suffices to characterize \(H_\lambda\) in the \((\fq(n),\fq(n))\)-Howe duality, since the natural identifications above then determine the highest weight vector spaces for all queer Howe dualities.

Unlike the general linear case, where the relevant highest weight space is one-dimensional, \(H_\lambda\) has dimension \(2^\ell\). We therefore use the actions of the Cartan subalgebras of the two queer Lie superalgebras to construct a natural irreducible \(\Cl_{2\ell}\)-module structure on \(H_\lambda\), giving a more precise description.

As in the classical case, the \((\fq(n),\fq(n))\)-Howe duality can be viewed as the subrepresentation \(\co(\cx(n))\subset \co(Q(n))\) in the regular representation on \(\co(Q(n))\). To obtain its singular vectors, we extend the regular actions on \(\co(Q(n))\) to \(\co(\Omega_n)\), where \(\Omega_n\) is the big cell associated with the Gauss decomposition of \(Q(n)\).

For \(r\in\bn\), let \(d_r:=\det(x_{ij})_{1\leq i,j\leq r}\), and let \(\cd\subset\cpinf\) be the multiplicative subset generated by all \(d_r\). We regard all the algebras \(\co(\Omega_n)\) as subalgebras of the common localization \(\cd^{-1}\cpinf\). The diagonal coordinate functions arising from the Gauss decompositions are compatible with these inclusions; we denote their common images by \(a_k,b_k\in\cd^{-1}\cpinf\). We obtain the following complete characterization of the singular vectors in \(\co(\Omega_n)\).

\begin{theoremA}
The \(\fn^-\oplus\fn^+\)-singular subspace of \(\co(\Omega_n)\) is given by
\[
\Sing(\co(\Omega_n))=\Gamma_n^{-1}(1\otimes\co(H_n)\otimes 1),
\]
where \(\Gamma_n\) denotes the Gauss decomposition isomorphism
\[
\co(\Omega_n)\overset{\sim}{\to}\co(U_n^-)\otimes\co(H_n)\otimes\co(U_n^+).
\]
More explicitly,
\[
\Sing(\co(\Omega_n))=
\Span\left\{
\prod_{i=1}^n a_i^{r_i}b_i^{\epsilon_i}
\;\middle|\;
r_i\in\bz,\ \epsilon_i\in\{0,1\},\ 1\leq i\leq n
\right\},
\]
where the elements \(a_k,b_k\in\cd^{-1}\cpinf\) for \(k\in\bn\) are given by
\[
\begin{pmatrix}
a_k&b_k\\b_k&a_k
\end{pmatrix}
=Z_{kk}-
(Z_{k1},\ldots,Z_{k,k-1})
(Z_{ij})_{1\leq i,j\leq k-1}^{-1}
(Z_{1k},\ldots,Z_{k-1,k})^t
\]
with
\[
Z_{ij}:=
\begin{pmatrix}
x_{ij}&y_{ij}\\y_{ij}&x_{ij}
\end{pmatrix}.
\]
\end{theoremA}

By considering the weight space decomposition of the singular subspace in \(\co(\Omega_n)\), we obtain the singular vectors in \(\co(\cx(n))\). Identifying \(\co(\cx(n))\) with \(S\bigl(\bc^{n^2|n^2}\bigr)\) and viewing it as a subspace of \(\cpinf\), we obtain the highest weight vectors in \(H_\lambda\).

\begin{theoremB}
Let \(\lambda=(\lambda_1,\ldots,\lambda_\ell)\) be a strict partition. Then the set
\[
\left\{\prod_{k=1}^{\ell}a_k^{\lambda_k-\delta_k}b_k^{\delta_k}\;\middle|\;\delta_k\in\{0,1\},\ 1\leq k\leq\ell\right\}
\]
is a basis of \(H_\lambda\).
\end{theoremB}

The method developed above also applies to the \((\fgl(p|q),\fgl(m|n))\)-Howe duality in the case \(p=m\). In this case, the singular vectors in \(\co(\Omega)\) recover the explicit highest weight vectors obtained by Cheng and Wang \cite[Theorems~4.3 and~4.4]{MR1847665}.

\begin{theoremC}
For fixed \(m\), let \(\lambda\) be a partition. Define
\[
w_\lambda:=\prod_{i=1}^m \alpha_i^{\lambda_i}
\cdot\prod_{j=1}^{\lambda_{m+1}} \bigl(\beta_j^{(m)}\bigr)^{\langle\lambda'_j-m\rangle},
\]
where the elements \(\alpha_i,\beta_j^{(m)}\in\cd_m^{-1}\cpinf^{(m)}\) are characterized by
\begin{align*}
&\prod_{i=1}^r \alpha_i=\det\bfX_{[r],[r]}&&(1\leq r\leq m),\\
&\prod_{j=1}^s \beta_j^{(m)}=\det\,(\bfY-\bfXi \bfX^{-1}\bfH)_{[s],[s]}&&(s\in\bn),
\end{align*}
where,
\begin{align*}
\bfX=&(x_{ij})_{1\leq i,j\leq m},
&\bfH=&(\eta_{il})_{1\leq i\leq m,\ l\in\bn},\\
\bfXi=&(\xi_{kj})_{k\in\bn,\ 1\leq j\leq m},
&\bfY=&(y_{kl})_{k,l\in\bn}.
\end{align*}
Then \(H_\lambda^{(m)}=\bc w_\lambda\). In particular, for \(q,n\geq\lambda_{m+1}\), under the identification, \(w_\lambda\) is a highest weight vector for the \(L_{m|q}(\lambda)\otimes L_{m|n}(\lambda)\) summand in the \((\fgl(m|q),\fgl(m|n))\)-Howe duality.
\end{theoremC}

\subsection*{Structure of the paper}

The paper is organized as follows. Section~\ref{sec:prelim} collects the necessary preliminaries on Clifford algebras, Hopf superalgebras, and the queer supergroup \(Q(n)\). Section~\ref{sec:queer-howe} recalls the Howe duality for the queer Lie supaeralgebra and construct natural Clifford algebra module structures on the highest weight spaces. Section~\ref{sec:regular} constructs the regular representations on \(\co(\Omega)\). Section~\ref{sec:singular} determines the singular vectors in \(\co(\Omega)\) and proves the basis theorem for the highest weight vectors in Howe duality. Section~\ref{sec:gl-case} applies the same argument to the \((\fgl(m|q),\fgl(m|n))\)-Howe duality, and shows that the resulting highest weight vectors recover the explicit formulas of Cheng and Wang.

\subsection*{Notation}

Unless otherwise stated, we work over the field \(\bc\) of complex numbers throughout the paper. We write \(\bz\), \(\bz_+\), and \(\bn\) for the sets of integers, nonnegative integers, and positive integers, respectively. We denote by \(\delta_{ij}\) the Kronecker delta.

For a partition \(\lambda\), we denote by \(\lambda'\) its conjugate partition.

Set \(\bz_2=\{\bar{0},\bar{1}\}\). For a homogeneous element \(x\) of an associative superalgebra or Lie superalgebra, we denote by \(|x|\in\bz_2\) its parity. Whenever \(|x|\) occurs, \(x\) is always assumed to be homogeneous, and formulas involving parity are extended by linearity. We write \([-,-]\) for the supercommutator.

For a superalgebra \(R\), we write \(M_{s,t}(R)\) for the vector space of \(s\times t\) matrices with entries in \(R\), and \(M_n(R)\) for the superalgebra of \(n\times n\) matrices with entries in \(R\). A superalgebra \(R\) is called \emph{commutative} if \(xy=(-1)^{|x||y|}yx\) for \(x,y\in R\).

Let \(A\) and \(B\) be semisimple superalgebras. If \(M\) is a simple \(A\)-module of type \(\mathtt{Q}\) and \(N\) is a simple \(B\)-module of type \(\mathtt{Q}\), then the \(A\otimes B\)-module \(M\otimes N\) is the direct sum of two isomorphic copies of a simple module of type \(\mathtt{M}\), denoted by \(2^{-1}(M\otimes N)\).

For a matrix \(M\) and index sets \(I,J\), we denote by \(M_{I,J}\) the submatrix of \(M\) with rows indexed by \(I\) and columns indexed by \(J\). For a positive integer \(n\), we write \([n]=\{1,2,\ldots,n\}\). In particular, \(M_{[r],[r]}\) denotes the leading \(r\times r\) principal submatrix of \(M\).

\section{Preliminaries}\label{sec:prelim}

\subsection{Clifford algebras}

We begin with some basic facts about Clifford algebras that will be used later; see for instance \cite[Chapter~II]{MR60497}.

The \emph{Clifford algebra} \(\Cl_{2\ell}\) is the associative superalgebra with odd generators
\[
c_1^+,\ldots,c_\ell^+,\ c_1^-,\ldots,c_\ell^-
\]
and defining relations
\[
c_i^+c_j^+ + c_j^+c_i^+ = 0,\quad
c_i^-c_j^- + c_j^-c_i^- = 0,\quad
c_i^+c_j^- + c_j^-c_i^+ = \delta_{ij},
\quad (1\leq i,j\leq \ell).
\]
An irreducible \(\Cl_{2\ell}\)-module is called a \emph{spinor module}. The spinor module is unique up to isomorphism; we denote it by \(\cs\).

Let \(C^+\) and \(C^-\) be the subalgebras of \(\Cl_{2\ell}\) generated by \(c_1^+,\ldots,c_\ell^+\) and \(c_1^-,\ldots,c_\ell^-\), respectively. Set \(\psi_0:=c_1^-c_2^-\cdots c_\ell^-\). Then the minimal left ideal \(\Cl_{2\ell}\psi_0=C^+\psi_0\) realizes the spinor module. It has a basis given by
\[
\{c_{i_1}^+\cdots c_{i_r}^+\psi_0\mid
\ 1\leq i_1<\cdots<i_r\leq \ell,0\leq r\leq \ell\}.
\]
Thus, \(\dim\cs=2^\ell\).

For a spinor module \(\cs\), define
\[
\cs^{\pure}:=\{s\in \cs\mid c_i^-s=0,\ \forall 1\leq i\leq \ell\}.
\]
The nonzero elements of \(\cs^{\pure}\) are called \emph{pure spinors} with respect to \(c_1^-,\ldots,c_\ell^-\). The subspace \(\cs^{\pure}\) is one-dimensional; in the realization above, it is \(\bc\psi_0\). For any pure spinor \(\psi\) with respect to \(c_1^-,\ldots,c_\ell^-\), one has \(\cs=C^+\psi\), and a basis is given by
\begin{equation}\label{eq:basis}
\{c_{i_1}^+\cdots c_{i_r}^+\psi\mid
1\leq i_1<\cdots<i_r\leq \ell,\ 0\leq r\leq \ell\}.
\end{equation}

\subsection{Hopf superalgebras}

Here are some basic facts about Hopf superalgebras.

A \emph{Hopf superalgebra} is a Hopf algebra in the category of vector superspaces, with braiding given by the Koszul sign rule. For a Hopf superalgebra \(H\), we write
\[
\mu:H\otimes H\to H,\ \Delta:H\to H\otimes H,\ \iota:\bc\to H,\ \epsilon:H\to\bc,\ S:H\to H,
\]
for its multiplication, comultiplication, unit, counit, and antipode, respectively. We use sumless Sweedler notation throughout the paper.

Let \(H\) and \(K\) be Hopf superalgebras. An even linear map
\[
\lra{-}{-}:H\otimes K\to\bc
\]
is called a \emph{Hopf pairing} if, for \(x,y\in H\) and \(a,b\in K\),
\begin{align*}
&\lra{xy}{a}=\lra{x\otimes y}{\Delta(a)}=(-1)^{|y||a_{(1)}|}\lra{x}{a_{(1)}}\lra{y}{a_{(2)}},\\
&\lra{x}{ab}=\lra{\Delta(x)}{a\otimes b}=(-1)^{|x_{(2)}||a|}\lra{x_{(1)}}{a}\lra{x_{(2)}}{b},\\
&\lra{1}{a}=\epsilon(a),\quad\lra{x}{1}=\epsilon(x).
\end{align*}
These conditions imply antipode compatibility
\[
\lra{S(x)}{a}=\lra{x}{S(a)},\quad x\in H,a\in K.
\]

\subsection{Queer Lie superalgebra}

Let \(\bc^{n|n}\) denote the vector superspace of dimension \(n|n\). We define
\[
\fq(n):=\left\{
\begin{pmatrix}
A&B\\B&A
\end{pmatrix}
\in \fgl(n|n):
A,B\in M_n(\bc)
\right\}.
\]
This is a subalgebra of \(\fgl(n|n)\), called the \emph{queer Lie superalgebra}.

For \(1\leq i,j\leq n\), let \(e_{ij}\in M_n(\bc)\) be the matrix unit, with \(1\) in the \((i,j)\)-entry and \(0\) elsewhere. Define
\[
E_{ij}:=
\begin{pmatrix}
e_{ij}&0\\0&e_{ij}
\end{pmatrix},
\quad F_{ij}:=
\begin{pmatrix}
0&e_{ij}\\e_{ij}&0
\end{pmatrix}
\quad(1\leq i,j\leq n).
\]
Then \(\{E_{ij}\mid 1\leq i,j\leq n\}\) and \(\{F_{ij}\mid 1\leq i,j\leq n\}\) are bases of \(\fq(n)_{\bar 0}\) and \(\fq(n)_{\bar 1}\), respectively. Set
\begin{align*}
\fh_{\bar 0}:=&\Span\{E_{ii}\mid 1\leq i\leq n\},\quad
\fh_{\bar 1}:=\Span\{F_{ii}\mid 1\leq i\leq n\},\quad
\fh:=\fh_{\bar 0}\oplus\fh_{\bar 1},\\
\fn^+:=&\Span\{E_{ij},F_{ij}\mid 1\leq i<j\leq n\},\quad
\fn^-:=\Span\{E_{ij},F_{ij}\mid 1\leq j<i\leq n\}.
\end{align*}
Then \(\fh\) is the \emph{Cartan subalgebra} of \(\fq(n)\), and
\[
\fq(n)=\fn^-\oplus\fh\oplus\fn^+
\]
is the standrad triangular decomposition of \(\fq(n)\).

Let \(\fb:=\fh\oplus\fn^+\) be the \emph{Borel subalgebra} of \(\fq(n)\). For \(\lambda=(\lambda_1,\ldots,\lambda_n)\in\bc^n\), we regard \(\lambda\) as an element of \(\fh_{\bar 0}^*\) by setting
\[
\lambda(E_{ii})=\lambda_i\quad (1\leq i\leq n).
\]
If \(M\) is a \(\fq(n)\)-module, its \(\lambda\)-weight space is defined by
\[
M_\lambda:=\{m\in M\mid E_{ii}m=\lambda_i m,1\leq i\leq n\}.
\]

Unlike the case of semisimple Lie algebras, the Cartan subalgebra \(\fh\) of \(\fq(n)\) is non-abelian. For each \(\lambda\in\fh_{\bar 0}^*\), the action of \(\fh_{\bar 1}\) on a highest weight space is governed by the Clifford structure associated with \(\lambda\); see \cite[Section~1.5.4]{MR3012224}. Let \(W_\lambda\) denote the irreducible \(\fh\)-module with \(\fh_{\bar 0}\)-weight \(\lambda\). We regard \(W_\lambda\) as a \(\fb\)-module by letting \(\fn^+\) act trivially.

The Verma module of highest weight \(\lambda\) is defined by
\[
M(\lambda):=U(\fq(n))\otimes_{U(\fb)}W_\lambda.
\]
It has a unique irreducible quotient, denoted by \(L(\lambda)\). Moreover, \(L(\lambda)\) is finite-dimensional if and only if
\[
\lambda_i-\lambda_{i+1}\in\bz_{+}
\quad\text{and}\quad
\lambda_i=\lambda_{i+1}\Rightarrow \lambda_i=0,
\]
for all \(1\leq i\leq n-1\).

\subsection{Queer supergroup}

\subsubsection{}

Let \(\mathbf{SAlg}\) denote the category of commutative superalgebras, and let \(\mathbf{Set}\) denote the category of sets. If \(X=\spec R\) is an affine superscheme, we use its functor of points
\[
h_X:\mathbf{SAlg}\to\mathbf{Set},\quad
A\mapsto \Hom_{\mathbf{SAlg}}(R,A),\quad
h_X(\phi)(f)=\phi\circ f.
\]
We write \(\co(X)\) for the coordinate superalgebra of \(X\). In this paper, a supergroup means an affine supergroup scheme, namely an affine superscheme whose functor of points is group-valued. Equivalently, its coordinate superalgebra is a commutative Hopf superalgebra; see \cite{MR2840967} for more details.

Let \(\cx(n)\) be the affine superspace isomorphic to \(\ba^{n^2|n^2}\), defined by the following functor of points. For \(A\in\mathbf{SAlg}\), set
\[
\cx(n)(A):=\left\{
\begin{pmatrix}
a&b\\b&a
\end{pmatrix}
\;\middle|\;a\in M_n(A_{\bar 0}),\ b\in M_n(A_{\bar 1})\right\}.
\]
Its coordinate superalgebra is
\[
\co(\cx(n))=\bc[x_{ij},y_{ij}\mid 1\leq i,j\leq n],
\]
where the \(x_{ij}\) are even indeterminates and the \(y_{ij}\) are odd indeterminates.

The queer supergroup \(Q(n)\) is the subfunctor of \(\cx(n)\) defined by
\[
Q(n)(A):=\left\{
\begin{pmatrix}
a&b\\b&a
\end{pmatrix}
\in\cx(n)(A)\;\middle|\; a\ \text{is invertible}\right\}.
\]
It is a supergroup whose coordinate Hopf superalgebra is
\[
\co(Q(n))=\bc[x_{ij},y_{ij}\mid 1\leq i,j\leq n]_{d_n},
\]
where \(d_n:=\det(x_{ij})_{1\leq i,j\leq n}\). The Hopf superalgebra structure on \(\co(Q(n))\) is determined by the multiplication, identity, and inverse in the functor \(Q(n)\). On \(x_{ij},y_{ij}\,(1\leq i,j\leq n)\), the structure maps are given as follows. The comultiplication is given by
\[
\Delta(x_{ij})=\sum_{k=1}^n (x_{ik}\otimes x_{kj}+y_{ik}\otimes y_{kj}),\quad
\Delta(y_{ij})=\sum_{k=1}^n (x_{ik}\otimes y_{kj}+y_{ik}\otimes x_{kj}).
\]
The counit is given by
\[
\epsilon(x_{ij})=\delta_{ij},\quad \epsilon(y_{ij})=0.
\]
The antipode is given by
\[
S(x_{ij})=((X-YX^{-1}Y)^{-1})_{ij},\quad S(y_{ij})=(-X^{-1}Y(X-YX^{-1}Y)^{-1})_{ij},
\]
where
\[
X=(x_{ij})_{1\leq i,j\leq n},\,Y=(y_{ij})_{1\leq i,j\leq n}\in M_n(\co(Q(n))).
\]

\subsubsection{}

We define three subgroups of \(Q(n)\): the upper unipotent subgroup \(U^+\), the lower unipotent subgroup \(U^-\), and the diagonal subgroup \(H\). For \(A\in\mathbf{SAlg}\), every element of \(Q(n)(A)\) has the form
\[
\begin{pmatrix}
a & b\\b & a
\end{pmatrix},
\quad a\in \GL_n(A_{\bar 0}),\quad b\in M_n(A_{\bar 1}).
\]
We define their \(A\)-points by imposing the following conditions:
\[
\begin{array}{lll}
U^+(A):& a\text{ is upper unipotent triangular},& b\text{ is strictly upper triangular},\\
U^-(A):& a\text{ is lower unipotent triangular},& b\text{ is strictly lower triangular},\\
H(A):& a\text{ is diagonal},& b\text{ is diagonal}.
\end{array}
\]

The coordinate Hopf superalgebras of \(U^\pm\) and \(H\) are obtained as quotients of \(\co(Q(n))\). Let
\begin{align*}
I_+&=\left(
\{x_{ii}-1\mid 1\leq i\leq n\}
\cup \{x_{ji},y_{ji}\mid 1\leq i<j\leq n\}
\cup \{y_{ii}\mid 1\leq i\leq n\}
\right),\\
I_-&=\left(
\{x_{ii}-1\mid 1\leq i\leq n\}
\cup \{x_{ij},y_{ij}\mid 1\leq i<j\leq n\}
\cup \{y_{ii}\mid 1\leq i\leq n\}
\right),\\
I_H&=\left(
\{x_{ij},x_{ji},y_{ij},y_{ji}\mid 1\leq i<j\leq n\}
\right)
\end{align*}
be the corresponding ideals of \(\co(Q(n))\). Then
\[
\co(U^+)=\co(Q(n))/I_+,\quad
\co(U^-)=\co(Q(n))/I_-,\quad
\co(H)=\co(Q(n))/I_H.
\]

For \(1\leq k\leq n\), set \(d_k:=\det(x_{ij})_{1\leq i,j\leq k}\). We denote the principal open affine subscheme of \(Q(n)\) defined by \(d_1\cdots d_{n-1}\) by \(\Omega\), and call it the \emph{big cell}. Equivalently, \(\Omega\) is the principal open affine subscheme of \(\cx(n)\) defined by \(d_1\cdots d_n\); so its coordinate superalgebra is
\[
\co(\Omega)=\bc[x_{ij},y_{ij}\mid 1\leq i,j\leq n]_{d_1\cdots d_n}.
\]
Then the corresponding \emph{Gauss decomposition} is as follows.

\begin{theorem}[{cf. \cite[Theorem~3.5]{MR1973576}}]\label{thm:gauss}
The multiplication morphism of \(Q(n)\) restricts to an isomorphism of affine superschemes
\[
\gamma:U^-\times H\times U^+\overset{\sim}{\to}\Omega.
\]
We denote the pullback of \(\gamma\) by
\[
\Gamma:\co(\Omega)\overset{\sim}{\to}\co(U^-)\otimes \co(H)\otimes\co(U^+).
\]
\end{theorem}

\subsubsection{}

We denote by \(I_e:=\ker\epsilon\) the \emph{augmentation ideal} of \(\co(Q(n))\). The \emph{algebra of distributions} on \(Q(n)\) is defined by
\[
\Dist(Q(n)):=\bigcup_{r\geq 0}(\co(Q(n))/I_e^{r+1})^*
\subseteq \co(Q(n))^*.
\]
It inherits a Hopf superalgebra structure from \(\co(Q(n))\) by duality. The Lie superalgebra \(\Lie(Q(n)):=(I_e/I_e^2)^*\) embeds naturally into \(\Dist(Q(n))\), and its bracket is the restriction of the supercommutator in \(\Dist(Q(n))\).

The residue classes of the elements in
\[
\{x_{ij}-\delta_{ij},\, y_{ij}\mid 1\leq i,j\leq n\}
\]
constitute a basis of \(I_e/I_e^2\). Let
\[
\{E'_{ij},\, F'_{ij}\mid 1\leq i,j\leq n\}
\]
be the corresponding dual basis, viewed as a basis of \(\Lie(Q(n))\). They satisfy
\begin{align*}
[E'_{ij},E'_{kl}]=&\delta_{jk}E'_{il}-\delta_{li}E'_{kj},&&[E'_{ij},F'_{kl}]=\delta_{jk}F'_{il}-\delta_{li}F'_{kj},\\
[F'_{ij},F'_{kl}]=&-\delta_{jk}E'_{il}-\delta_{li}E'_{kj}&&(1\leq i,j,k,l\leq n).
\end{align*}
Hence, the assignment
\[
E_{ij}\mapsto E'_{ij},\quad F_{ij}\mapsto \sqrt{-1} F'_{ij}\quad(1\leq i,j\leq n)
\]
defines a Lie superalgebra isomorphism \(\fq(n)\overset{\sim}{\to}\Lie(Q(n))\). Via this isomorphism, we identify \(\Lie(Q(n))\) with \(\fq(n)\).

The subgroups \(U^+\), \(U^-\), and \(H\) are closed in \(Q(n)\), and their Lie superalgebras are subalgebras of \(\Lie(Q(n))=\fq(n)\). More explicitly, we have the following descriptions.
\begin{align*}
\Lie(U^+)=&\Span\{E'_{ij},F'_{ij}\mid 1\leq i<j\leq n\}=\fn^+,\\
\Lie(U^-)=&\Span\{E'_{ji},F'_{ji}\mid 1\leq i<j\leq n\}=\fn^-,\\
\Lie(H)=&\Span\{E'_{ii},F'_{ii}\mid 1\leq i\leq n\}=\fh.
\end{align*}

By Theorem~4.4 of \cite{MR1973576}, the embedding \(\fq(n)\overset{\sim}{\to}\Lie(Q(n))\subseteq \Dist(Q(n))\) induces an isomorphism of Hopf superalgebras
\[
U(\fq(n))\overset{\sim}{\to}\Dist(Q(n)).
\]
Through this isomorphism, we identify \(U(\fq(n))\) with the Hopf superalgebra \(\Dist(Q(n))\). The evaluation map then defines a Hopf pairing
\[
\lra{-}{-}:\co(Q(n))\otimes U(\fq(n))\to\bc,\quad a\otimes u\mapsto u(a).
\]

\begin{proposition}\label{prop:nond}
The Hopf pairing defined by the evaluation map is nondegenerate.
\end{proposition}

\begin{proof}
Identifying \(U(\fq(n))\) with \(\Dist(Q(n))\), right nondegeneracy follows from the inclusion \(\Dist(Q(n))\subseteq\co(Q(n))^*\). For left nondegeneracy, Brundan--Kleshchev prove in Section~5 of \cite{MR1973576} that the evaluation map
\[
\co(Q(n))\to \Dist(Q(n))^*,\quad
f\mapsto [u\mapsto u(f)]
\]
is injective. Hence, the pairing separates elements of \(\co(Q(n))\).
\end{proof}

The same argument applies to the closed subgroups \(U^\pm\) and \(H\). Namely, after identifying
\[
U(\fn^\pm)\overset{\sim}{\to} \Dist(U^\pm),\quad U(\fh)\overset{\sim}{\to} \Dist(H),
\]
the evaluation maps define nondegenerate Hopf pairings
\[
\lra{-}{-}:\co(U^\pm)\otimes U(\fn^\pm)\to\bc,\quad
\lra{-}{-}:\co(H)\otimes U(\fh)\to\bc.
\]

\section{Howe duality for the queer Lie superalgebra}\label{sec:queer-howe}

We first recall the Howe duality for the queer Lie superalgebra. Let \(\bc^{m|m}\) and \(\bc^{n|n}\) be the natural representations of \(\fq(m)\) and \(\fq(n)\), respectively. We write \(\bc^{mn|mn}\) for \(2^{-1}(\bc^{m|m}\otimes\bc^{n|n})\) as a \(\fq(m)\oplus\fq(n)\)-module.

\begin{theorem}[{see \cite[Theorem~3.1]{MR1786858},\cite[Theorem~9.1]{MR1827078}}]\label{thm:howe-duality}
As a \(\fq(m)\oplus\fq(n)\)-module, the symmetric superalgebra \(S(\bc^{mn|mn})\) admits the multiplicity-free decomposition
\[
S(\bc^{mn|mn})\cong
\bigoplus_{\substack{\lambda\in\mathcal{SP}\\ \ell(\lambda)\leq \min(m,n)}}
2^{-\delta(\lambda)}L_m(\lambda)\otimes L_n(\lambda),
\]
where \(L_r(\lambda)\) denotes the irreducible \(\fq(r)\)-module of highest weight \(\lambda\), and \(\delta(\lambda)\) is the parity of \(\ell(\lambda)\).
\end{theorem}

We now choose explicit coordinates on \(S(\bc^{mn|mn})\), identifying it with a polynomial superalgebra on which the queer Lie superalgebras act by differential operators.

Let \(\{e_1,\ldots,e_m,f_1,\ldots,f_m\}\) and \(\{u_1,\ldots,u_n,v_1,\ldots,v_n\}\) be standard bases of the natural representations \(\bc^{m|m}\) of \(\fq(m)\) and \(\bc^{n|n}\) of \(\fq(n)\), respectively, where \(e_i,u_j\) are even and \(f_i,v_j\) are odd. The tensor product \(\bc^{m|m}\otimes\bc^{n|n}\) is the direct sum of two isomorphic irreducible \(\fq(m)\oplus\fq(n)\)-summands. We realize \(\bc^{mn|mn}\) as the summand having the following homogeneous basis:
\[
x_{ij}:=e_i\otimes u_j-\sqrt{-1}f_i\otimes v_j,\quad
y_{ij}:=f_i\otimes u_j-\sqrt{-1}e_i\otimes v_j
\quad (1\leq i\leq m,\ 1\leq j\leq n).
\]
Here \(x_{ij}\) is even and \(y_{ij}\) is odd, so \(S(\bc^{mn|mn})\) is identified with the polynomial superalgebra generated by the even variables \(x_{ij}\) and the odd variables \(y_{ij}\). Under this identification, the actions of \(\fq(m)\) and \(\fq(n)\) are given by
\begin{align*}
\varrho_1(E_{ij})&=\sum_{k=1}^n\biggl(x_{ik}\frac{\partial}{\partial x_{jk}}+y_{ik}\frac{\partial}{\partial y_{jk}}\biggr),&\varrho_1(F_{ij})&=\sum_{k=1}^n\biggl(y_{ik}\frac{\partial}{\partial x_{jk}}+x_{ik}\frac{\partial}{\partial y_{jk}}\biggr),\\
\varrho_2(E_{ij})&=\sum_{k=1}^m\biggl(x_{ki}\frac{\partial}{\partial x_{kj}}+y_{ki}\frac{\partial}{\partial y_{kj}}\biggr),&\varrho_2(F_{ij})&=\sqrt{-1}\sum_{k=1}^m\biggl(y_{ki}\frac{\partial}{\partial x_{kj}}-x_{ki}\frac{\partial}{\partial y_{kj}}\biggr).
\end{align*}

Next, we consider the directed system of polynomial superalgebras \(S(\bc^{mn|mn})\), indexed by \(m,n\in\bn\), under the natural inclusions. Its direct limit is the polynomial superalgebra
\[
\cpinf:=\bc[x_{ij},y_{ij}\mid i,j\in\bn].
\]
Let \(\fq(\infty):=\varinjlim_r\fq(r)\) be the infinite queer Lie superalgebra obtained as the direct limit of the queer Lie superalgebras \(\fq(r)\) under the natural inclusions. The formulas
\begin{align*}
\varrho_1(E_{ij})&=\sum_{k\geq 1}\biggl(x_{ik}\frac{\partial}{\partial x_{jk}}+y_{ik}\frac{\partial}{\partial y_{jk}}\biggr),&
\varrho_1(F_{ij})&=\sum_{k\geq 1}\biggl(y_{ik}\frac{\partial}{\partial x_{jk}}+x_{ik}\frac{\partial}{\partial y_{jk}}\biggr),\\
\varrho_2(E_{ij})&=\sum_{k\geq 1}\biggl(x_{ki}\frac{\partial}{\partial x_{kj}}+y_{ki}\frac{\partial}{\partial y_{kj}}\biggr),&
\varrho_2(F_{ij})&=\sqrt{-1}\sum_{k\geq 1}\biggl(y_{ki}\frac{\partial}{\partial x_{kj}}-x_{ki}\frac{\partial}{\partial y_{kj}}\biggr),
\end{align*}
define two actions of \(\fq(\infty)\) on \(\cpinf\), since every polynomial involves only finitely many variables.

Under the canonical map to this direct limit, each \(S(\bc^{mn|mn})\) is identified with the subalgebra of \(\cpinf\) generated by
\[
x_{ij},y_{ij}\quad (1\leq i\leq m,\ 1\leq j\leq n).
\]
With this identification, the restrictions of the two \(\fq(\infty)\)-actions along the inclusions \(\fq(m),\fq(n)\subset\fq(\infty)\) coincide with the actions \(\varrho_1\) and \(\varrho_2\) in the Howe duality above, respectively.

\begin{definition}\label{def:Hlambda}
Let \(\lambda=(\lambda_1,\ldots,\lambda_\ell)\) be a strict partition of length \(\ell\). For \(m,n\geq \ell\), let \(H_\lambda^{m,n}\) denote the space of highest weight vectors in the \(2^{-\delta(\lambda)}L_m(\lambda)\otimes L_n(\lambda)\)-summand of \(S(\bc^{mn|mn})\) occurring in Theorem~\ref{thm:howe-duality}. Set
\[
H_\lambda:=H_\lambda^{\ell,\ell}\subset S\bigl(\bc^{\ell^2|\ell^2}\bigr)\subset\cpinf.
\]
\end{definition}

\begin{remark}\label{rmk:reduction-square}
Let \(m,n\geq \ell\). Under the embeddings above, the inclusion
\[
S\bigl(\bc^{\ell^2|\ell^2}\bigr)\subset S(\bc^{mn|mn})\subset\cpinf
\]
maps \(H_\lambda\) into \(H_\lambda^{m,n}\). Indeed, the relevant actions are obtained by restricting the two \(\fq(\infty)\)-actions on \(\cpinf\). The standard basis elements of \(\fn^+\) already contained in the two copies of \(\fq(\ell)\) act as in the \((\fq(\ell),\fq(\ell))\)-case. The other standard basis elements of \(\fn^+\) in \(\fq(m)\) or \(\fq(n)\) differentiate with respect to variables whose row or column index is greater than \(\ell\), and therefore annihilate \(S\bigl(\bc^{\ell^2|\ell^2}\bigr)\). Thus, the image of \(H_\lambda\) consists of highest weight vectors of weight \((\lambda,\lambda)\) in \(S(\bc^{mn|mn})\).

Since \(\dim H_\lambda=\dim H_\lambda^{m,n}=2^\ell\) (cf. \cite[Section~1.5.4]{MR3012224}), this inclusion identifies \(H_\lambda\) with \(H_\lambda^{m,n}\). We shall therefore regard all spaces \(H_\lambda^{m,n}\) as the same subspace \(H_\lambda\) of \(\cpinf\). In particular, for the purpose of characterizing these highest weight vectors, it suffices to consider the \((\fq(r),\fq(r))\)-case.
\end{remark}

Since \(H_\lambda\) is the highest weight space of the \(\lambda\)-summand, it is stable under the restrictions of both Cartan actions \(\varrho_1|_{\fh}\) and \(\varrho_2|_{\fh}\). The following proposition uses these actions to endow \(H_\lambda\) with the structure of a spinor module for a natural Clifford algebra action.

\begin{proposition}\label{prop:cli}
Let \(\lambda=(\lambda_1,\ldots,\lambda_\ell)\) be a strict partition. For the generators \(c_i^\pm\) of \(\Cl_{2\ell}\), define endomorphisms \(\Theta(c_i^\pm)\) of \(H_\lambda\) by
\begin{align*}
\Theta(c_i^+)v
&=\frac{1}{2}(\varrho_1(F_{ii})-\sqrt{-1}\,\varrho_2(F_{ii}))v,\\
\Theta(c_i^-)v
&=\frac{1}{2\lambda_i}
(\varrho_1(F_{ii})+\sqrt{-1}\varrho_2(F_{ii}))v
\quad (v\in H_\lambda).
\end{align*}
Then the assignment \(c_i^\pm\mapsto \Theta(c_i^\pm)\) extends uniquely to an action of \(\Cl_{2\ell}\) on \(H_\lambda\). With this action, \(H_\lambda\) is an irreducible \(\Cl_{2\ell}\)-module, that is, a spinor module.
\end{proposition}

\begin{proof}
It remains only to check that the endomorphisms \(\Theta(c_i^\pm)\) satisfy the Clifford relations; the irreducibility then follows from \(\dim H_\lambda=2^\ell\) and the classification of irreducible \(\Cl_{2\ell}\)-modules. Set
\[
A_i:=\varrho_1(F_{ii}),\quad B_i:=\varrho_2(F_{ii})
\quad (1\leq i\leq \ell).
\]
Then
\[
[A_i,B_j]=0,\quad [A_i,A_j]=[B_i,B_j]=2\delta_{ij}\lambda_i\id
\quad (1\leq i,j\leq\ell),
\]
and
\[
\Theta(c_i^+)=\frac{1}{2}(A_i-\sqrt{-1}B_i),\quad
\Theta(c_i^-)=\frac{1}{2\lambda_i}(A_i+\sqrt{-1}B_i).
\]
Hence,
\begin{align*}
[\Theta(c_i^+),\Theta(c_j^+)]
&=\frac{1}{4}[A_i-\sqrt{-1}B_i,A_j-\sqrt{-1}B_j]\\
&=\frac{1}{4}([A_i,A_j]-[B_i,B_j])=0,\\
[\Theta(c_i^-),\Theta(c_j^-)]
&=\frac{1}{4\lambda_i\lambda_j}
[A_i+\sqrt{-1}B_i,A_j+\sqrt{-1}B_j]\\
&=\frac{1}{4\lambda_i\lambda_j}
([A_i,A_j]-[B_i,B_j])=0,\\
[\Theta(c_i^+),\Theta(c_j^-)]
&=\frac{1}{4\lambda_j}
[A_i-\sqrt{-1}B_i,A_j+\sqrt{-1}B_j]\\
&=\frac{1}{4\lambda_j}
([A_i,A_j]+[B_i,B_j])=\delta_{ij}\id.
\end{align*}
These are precisely the defining relations of \(\Cl_{2\ell}\).
\end{proof}

\section{Regular representations on coordinate superalgebras}\label{sec:regular}

\subsection{Regular representations on \texorpdfstring{\(\co(Q(n))\)}{O(Q(n))}}

We construct the regular representations on \(\co(Q(n))\) using its Hopf superalgebra structure.

\begin{proposition}\label{prop:rho1}
Define a linear map \(\rho_1:U(\fq(n))\to\End(\co(Q(n)))\) by
\[
\rho_1(x)f=(-1)^{|x|}\lra{f_{(1)}}{S^{-1}(x)}f_{(2)},\quad x\in U(\fq(n)),f\in\co(Q(n)).
\]
Then \(\rho_1\) defines a \(\fq(n)\)-module structure on \(\co(Q(n))\).
\end{proposition}

\begin{proof}
It suffices to prove that
\[
\rho_1(xy)=\rho_1(x)\rho_1(y),
\]
for \(x,y\in U(\fq(n))\), and that \(\rho_1(1)=\id_{\co(Q(n))}\). Let \(f\in\co(Q(n))\). Then
\begin{align*}
\rho_1(x)\rho_1(y)f=&(-1)^{|y|}\rho_1(x)(\lra{f_{(1)}}{S^{-1}(y)}f_{(2)})\\
=&(-1)^{|y|+|x|}\lra{f_{(1)}}{S^{-1}(y)}\lra{f_{(2)}}{S^{-1}(x)}f_{(3)}\\
=&(-1)^{|xy|}\lra{f_{(1)}}{S^{-1}(xy)}f_{(2)}\\
=&\rho_1(xy)f.
\end{align*}
Moreover,
\[
\rho_1(1)f=\lra{f_{(1)}}{S^{-1}(1)}f_{(2)}=\epsilon(f_{(1)})f_{(2)}=f.
\]
Therefore, \(\rho_1\) defines a \(\fq(n)\)-module structure on \(\co(Q(n))\).
\end{proof}

\begin{proposition}\label{prop:rho2}
Define a linear map \(\rho_2:U(\fq(n))\to\End(\co(Q(n)))\) by
\[
\rho_2(x)f=(-1)^{|x||f|} f_{(1)}\lra{f_{(2)}}{x},\quad x\in U(\fq(n)),f\in\co(Q(n)).
\]
Then \(\rho_2\) defines a \(\fq(n)\)-module structure on \(\co(Q(n))\).
\end{proposition}

\begin{proof}
It suffices to prove that
\[
\rho_2(xy)=\rho_2(x)\rho_2(y),
\]
for \(x,y\in U(\fq(n))\), and that \(\rho_2(1)=\id_{\co(Q(n))}\). Let \(f\in\co(Q(n))\). Then
\begin{align*}
\rho_2(x)\rho_2(y)f=&(-1)^{|y||f|}\rho_2(x)(f_{(1)}\lra{f_{(2)}}{y})\\
=&(-1)^{|y||f|+|x|(|y|+|f|)}f_{(1)}\lra{f_{(2)}}{x}\lra{f_{(3)}}{y}\\
=&(-1)^{|xy||f|}f_{(1)}\lra{f_{(2)}}{xy}\\
=&\rho_2(xy)f.
\end{align*}
Moreover,
\[
\rho_2(1)f=f_{(1)}\lra{f_{(2)}}{1}=f_{(1)}\epsilon(f_{(2)})=f.
\]
Therefore, \(\rho_2\) defines a \(\fq(n)\)-module structure on \(\co(Q(n))\).
\end{proof}

The representations defined in Propositions~\ref{prop:rho1} and~\ref{prop:rho2} are called the \emph{regular representations} on \(\co(Q(n))\). We verify that they define a \(\fq(n)\oplus\fq(n)\)-module structure on \(\co(Q(n))\).

\begin{proposition}\label{prop:rho12}
The linear map
\[
\rho_1\oplus\rho_2:\fq(n)\oplus\fq(n)\to\End(\co(Q(n)))
\]
defines a \(\fq(n)\oplus\fq(n)\)-module structure on \(\co(Q(n))\).
\end{proposition}

\begin{proof}
It remains only to verify that the two actions supercommute. That is, for \(x,y\in\fq(n)\), we need to show that
\[
\rho_1(x)\rho_2(y)-(-1)^{|x||y|}\rho_2(y)\rho_1(x)=0.
\]
For \(f\in\co(Q(n))\), we have
\begin{align*}
&\bigl[\rho_1(x)\rho_2(y)-(-1)^{|x||y|}\rho_2(y)\rho_1(x)\bigr]f\\
=&(-1)^{|y||f|}\rho_1(x)(f_{(1)}\lra{f_{(2)}}{y})-(-1)^{|x||y|+|x|}\rho_2(y)(\lra{f_{(1)}}{S^{-1}(x)}f_{(2)})\\
=&(-1)^{|y||f|+|x|}(\lra{f_{(1)}}{S^{-1}(x)}f_{(2)}\lra{f_{(3)}}{y}-\lra{f_{(1)}}{S^{-1}(x)}f_{(2)}\lra{f_{(3)}}{y})\\
=&0.
\end{align*}
Therefore, the supercommutator vanishes, and hence \(\rho_1\oplus\rho_2\) defines a \(\fq(n)\oplus\fq(n)\)-module structure on \(\co(Q(n))\).
\end{proof}

\begin{remark}\label{rmk:dualmod}
Propositions~\ref{prop:rho1} and~\ref{prop:rho2} are consequences of the following general results. Let \(H\) and \(K\) be braided Hopf algebras with invertible antipodes; see \cite{MR1289422}. Let
\[
p: H\otimes K\to \bone
\]
be a Hopf pairing between \(H\) and \(K\), where \(\bone\) is the unit object. We denote by \(\Psi_{V,W}:V\otimes W\to W\otimes V\) the braiding isomorphism.

If \((M,\delta_1)\) is a left \(H\)-comodule with structure map
\[
\delta_1: M\to H\otimes M,
\]
then \(M\) has a left \(K\)-module structure defined by
\[
\rho_1
=(p\otimes\id_M)\circ(\Psi_{H,K}^{-1}\otimes \id_M)\circ(S_K^{-1}\otimes\delta_1):\ K\otimes M\to M.
\]
Similarly, if \((M,\delta_2)\) is a right \(H\)-comodule with structure map
\[
\delta_2: M\to M\otimes H,
\]
then \(M\) has a left \(K\)-module structure defined by
\[
\rho_2
=(\id_M\otimes p)\circ\Psi_{M\otimes H,K}^{-1}\circ(\id_K\otimes\delta_2): K\otimes M\to M.
\]

Moreover, let \(H_i\) and \(K_i\) be braided Hopf algebras equipped with Hopf pairings
\[
p_i: H_i\otimes K_i\to \bone\quad (i=1,2).
\]
If \(M\) is an \((H_1,H_2)\)-bicomodule with structure maps
\[
\delta_1:M\to H_1\otimes M,\quad \delta_2:M\to M\otimes H_2,
\]
then the induced left \(K_1\)- and \(K_2\)-module structures on \(M\) braided commute, in the sense that
\[
\rho_1\circ(\id_{K_1}\otimes\rho_2)=\rho_2\circ(\id_{K_2}\otimes\rho_1)\circ(\Psi_{K_1,K_2}\otimes\id_M): K_1\otimes K_2\otimes M\to M.
\]

In the setting considered in this section, we take
\[
H_1=H_2=\co(Q(n)),\quad K_1=K_2=U(\fq(n)),\quad p_1=p_2=\lra{-}{-},\quad M=\co(Q(n)),
\]
where \(M=\co(Q(n))\) is regarded as an \((H_1,H_2)\)-bicomodule via its coproduct. Under this condition, the two induced module structures are precisely the regular representations \(\rho_1\) and \(\rho_2\) in Propositions~\ref{prop:rho1} and~\ref{prop:rho2}, and their braided commutativity gives Proposition~\ref{prop:rho12}.
\end{remark}

\subsection{Regular representations on \texorpdfstring{\(\co(\Omega)\)}{O(Ω)}}\label{sec:regular-big-cell}

We now extend the regular representations on \(\co(Q(n))\) to the coordinate superalgebra \(\co(\Omega)\) of the big cell. We first record some basic properties of superderivations.

A linear map \(D:R\to R\) is called a \emph{superderivation} if it satisfies the super Leibniz rule
\[
D(ab)=D(a)b+(-1)^{|D||a|}aD(b)
\]
for \(a,b\in R\), where \(|D|\in\bz_2\) denotes the parity of \(D\). The vector superspace of all superderivations on \(R\) is denoted by \(\Der(R)\). It is a subalgebra of Lie superalgebra \(\End(R)\), whose bracket is given by the supercommutator.

\begin{lemma}\label{lem:der}
Let \(R\) be a commutative superalgebra, and let \(S\subset R_{\bar 0}\) be a multiplicative subset. Then every superderivation \(D:R\to R\) extends uniquely to a superderivation \(D':S^{-1}R\to S^{-1}R\).
\end{lemma}

\begin{proof}
For ordinary derivations, this is the classical localization property; see
\cite[\href{https://stacks.math.columbia.edu/tag/07PD}{Tag~07PD}]{stacks-project}.
Since the elements of \(S\) are even, the same argument applies in the commutative superalgebra setting, with the usual Koszul sign rule. Explicitly, for \(D\), the extension is defined by
\[
D'\biggl(\frac r s\biggr)=\frac{sD(r)-D(s)r}{s^2},
\quad r\in R,\ s\in S.\qedhere
\]
\end{proof}

\begin{corollary}\label{cor:der-localization}
Let \(R\) be a commutative superalgebra, and let \(S\subset R_{\bar 0}\) be a multiplicative subset. Then Lemma~\ref{lem:der} induces Lie superalgebra homomorphisms
\[
\lambda_S:\Der(R)\to \Der(S^{-1}R),\quad D\mapsto D',
\]
where \(D'\) denotes the unique superderivation extending \(D\) to \(S^{-1}R\).
\end{corollary}

\begin{proof}
For superderivations \(D_1,D_2\) on \(R\), both
\[
[\lambda_S(D_1), \lambda_S(D_2)]\quad\text{and}\quad\lambda_S([D_1,D_2])
\]
are superderivations of \(S^{-1}R\) extending \([D_1,D_2]\). Hence, they are the same by uniqueness.
\end{proof}

Then we verify that the regular actions on \(\co(Q(n))\) satisfy the super Leibniz rule.

\begin{proposition}\label{prop:der}
The actions \(\rho_1\) and \(\rho_2\) of \(\fq(n)\) on \(\co(Q(n))\) act by superderivations. That is,
\[
\rho_1(\fq(n)),\,\rho_2(\fq(n))\subseteq\Der(\co(Q(n))).
\]
\end{proposition}

\begin{proof}
For \(f,g\in\co(Q(n))\) and \(x,y\in\fq(n)\), we have
\begin{align*}
\rho_1(x)(fg)&=(-1)^{|x|}\lra{(fg)_{(1)}}{S^{-1}(x)}(fg)_{(2)}\\
&=(-1)^{|x|+|f_{(2)}||g_{(1)}|}\lra{f_{(1)}g_{(1)}}{S^{-1}(x)}f_{(2)}g_{(2)}\\
&=(-1)^{|x|+|f_{(2)}||g_{(1)}|}(\lra{f_{(1)}}{S^{-1}(x)}\epsilon(g_{(1)})+\epsilon(f_{(1)})\lra{g_{(1)}}{S^{-1}(x)})f_{(2)}g_{(2)}\\
&=\rho_1(x)(f)g+(-1)^{|x||f|}f\rho_1(x)(g),\\
\rho_2(y)(fg)&=(-1)^{|y||fg|}(fg)_{(1)}\lra{(fg)_{(2)}}{y}\\
&=(-1)^{|y||fg|+|f_{(2)}||g_{(1)}|}f_{(1)}g_{(1)}\lra{f_{(2)}g_{(2)}}{y}\\
&=(-1)^{|y||fg|+|f_{(2)}||g_{(1)}|}f_{(1)}g_{(1)}(\lra{f_{(2)}}{y}\epsilon(g_{(2)})+\epsilon(f_{(2)})\lra{g_{(2)}}{y})\\
&=\rho_2(y)(f)g+(-1)^{|y||f|}f\rho_2(y)(g).
\end{align*}
Hence, \(\rho_1(x),\,\rho_2(y)\in\Der(\co(Q(n)))\).
\end{proof}

\begin{remark}\label{rmk:der}
Proposition~\ref{prop:der} follows from the following algebra version of Remark~\ref{rmk:dualmod}. If \((M,\delta_1)\) is a left \(H\)-comodule algebra, then the induced module structure \(\rho_1\) makes \(M\) into a left \(K\)-module algebra. Similarly, if \((M,\delta_2)\) is a right \(H\)-comodule algebra, then the induced module structure \(\rho_2\) makes \(M\) into a left \(K\)-module algebra. In particular, primitive elements of \(K\) act on \(M\) by superderivations.
\end{remark}

Let \(\lambda_\Omega:\Der(\co(Q(n)))\to\Der(\co(\Omega))\) be the Lie superalgebra homomorphism obtained from Corollary~\ref{cor:der-localization} by localizing \(\co(Q(n))\) at \(d_1\cdots d_{n-1}\). Then the maps
\[
\lambda_\Omega\circ\rho_1,\,\lambda_\Omega\circ\rho_2:
\fq(n)\to\Der(\co(\Omega))
\]
define the regular representations on \(\co(\Omega)\), which we still denote by \(\rho_1\) and \(\rho_2\).

\subsection{Nilpotent regular representations}

We use left multiplication by \(U^-\) and right multiplication by \(U^+\) on \(\Omega\) to define the regular \(\fn^-\)- and \(\fn^+\)-actions on \(\co(\Omega)\).

Left multiplication by \(U^-\) on \(\Omega\) gives a morphism
\[
\mu_\Omega^-:U^-\times\Omega\to\Omega.
\]
Its pullback gives a left \(\co(U^-)\)-comodule structure map
\[
\delta_\Omega^-:\co(\Omega)\to\co(U^-)\otimes\co(\Omega),\quad
f\mapsto f_{(-1)}\otimes f_{(0)}.
\]
By Remark~\ref{rmk:dualmod}, the map \(\delta_\Omega^-\) induces an action of \(\fn^-\) on \(\co(\Omega)\), given by
\[
\eta^-(x)f=(-1)^{|x|}\lra{f_{(-1)}}{S^{-1}(x)}f_{(0)},
\quad x\in U(\fn^-),f\in\co(\Omega).
\]

Right multiplication by \(U^+\) on \(\Omega\) gives a morphism
\[
\mu_\Omega^+:\Omega\times U^+\to\Omega.
\]
Its pullback gives a right \(\co(U^+)\)-comodule structure map
\[
\delta_\Omega^+:\co(\Omega)\to\co(\Omega)\otimes\co(U^+),\quad
f\mapsto f_{(0)}\otimes f_{(1)}.
\]
The map \(\delta_\Omega^+\) induces an action of \(\fn^+\) on \(\co(\Omega)\), given by
\[
\eta^+(y)f=(-1)^{|y||f|}f_{(0)}\lra{f_{(1)}}{y},
\quad y\in U(\fn^+),\ f\in\co(\Omega).
\]

\begin{proposition}
Nilpotent subalgebras \(\fn^-\) and \(\fn^+\) act on \(\co(\Omega)\) by superderivations. More precisely,
\[
\eta^-(\fn^-),\,\eta^+(\fn^+)\subseteq \Der(\co(\Omega)).
\]
\end{proposition}

\begin{proof}
As in Proposition~\ref{prop:der}, this follows from Remark~\ref{rmk:der} applied to the two comodule algebra structures on \(\co(\Omega)\).
\end{proof}

By uniqueness of superderivation extensions, these actions coincide with the regular actions on \(\co(\Omega)\) defined in Subsection~\ref{sec:regular-big-cell}.

\begin{proposition}\label{prop:eta-rho}
For \(x\in\fn^-\) and \(y\in\fn^+\), we have
\[
\eta^-(x)=\rho_1(x),\quad\eta^+(y)=\rho_2(y).
\]
\end{proposition}

\begin{proof}
By construction, on the subalgebra \(\co(Q(n))\subset\co(\Omega)\), the actions induced by \(\delta_\Omega^-\) and \(\delta_\Omega^+\) coincide with the restrictions of the regular actions. Namely, for \(x\in\fn^-\) and \(y\in\fn^+\),
\[
\eta^-(x)|_{\co(Q(n))}=\rho_1(x)|_{\co(Q(n))},\quad
\eta^+(y)|_{\co(Q(n))}=\rho_2(y)|_{\co(Q(n))}.
\]
The algebra \(\co(\Omega)\) is a localization of \(\co(Q(n))\), and both sides are superderivations of \(\co(\Omega)\) agreeing on \(\co(Q(n))\). Hence the uniqueness of extensions implies the desired identities.
\end{proof}

\subsection{Diagonal regular representations}

We use left and right multiplications of \(H\) on \(\Omega\) to define the regular \(\fh\)-actions on \(\co(\Omega)\). Let
\[
\mu^-_H:H\times\Omega\to\Omega,\quad
\mu^+_H:\Omega\times H\to\Omega
\]
be the left and right multiplication morphisms, respectively. Denote their pullbacks by
\[
\delta_H^-:\co(\Omega)\to\co(H)\otimes\co(\Omega),\quad
\delta_H^+:\co(\Omega)\to\co(\Omega)\otimes\co(H).
\]
As in the nilpotent case, these comodule structure maps induce \(\fh\)-actions on \(\co(\Omega)\). For \(z\in U(\fh)\) and \(f\in\co(\Omega)\), define
\[
\eta_H^-(z)f=(-1)^{|z|}\lra{f_{(-1)}}{S^{-1}(z)}f_{(0)},\quad \eta_H^+(z)f=(-1)^{|z||f|}f_{(0)}\lra{f_{(1)}}{z},
\]
where \(\delta_H^-(f)=f_{(-1)}\otimes f_{(0)}\) and \(\delta_H^+(f)=f_{(0)}\otimes f_{(1)}\). 

The same argument as in the nilpotent case shows that these actions agree with the regular actions on \(\co(\Omega)\). That is, for \(x,y\in\fh\),
\[
\eta_H^-(x)=\rho_1(x),\quad\eta_H^+(y)=\rho_2(y).
\]

Let
\[
\alpha^-:H\times U^-\overset{\sim}{\to}U^-\times H,\quad
\alpha^+:U^+\times H\overset{\sim}{\to}H\times U^+
\]
be the isomorphisms induced by the Gauss decomposition. We write \((\alpha^-)^*\) and \((\alpha^+)^*\) for the corresponding pullback maps on coordinate superalgebras. 

For the left \(H\)-multiplication, we have the commutative diagram
\[
\begin{tikzcd}[column sep=large]
H\times U^-\times H\times U^+
\arrow[rr,"\alpha^-\times\id\times\id"]
\arrow[d,"\id\times\gamma"'] &
&
U^-\times H\times H\times U^+
\arrow[d,"\id\times\mu_H\times\id"] \\
H\times\Omega \arrow[r,"\mu^-_H"']&
\Omega &
U^-\times H\times U^+ \arrow[l,"\gamma"].
\end{tikzcd}
\]
It gives
\begin{equation}\label{eq:dh1}
(\id\otimes\,\Gamma)\circ\delta_H^-
=((\alpha^-)^*\otimes\id\otimes\id)\circ(\id\otimes\,\Delta\otimes\id)\circ\Gamma.
\end{equation}
Similarly, the commutative diagram for right \(H\)-multiplication gives
\[
(\Gamma\otimes\id)\circ\delta_H^+
=(\id\otimes\id\otimes(\alpha^+)^*)\circ(\id\otimes\,\Delta\otimes\id)\circ\Gamma.
\]

We define the regular \(\fh\)-actions on \(\co(H)\) by
\begin{align*}
\vartheta_H^-(z)g&=(-1)^{|z|}\lra{g_{(1)}}{S^{-1}(z)}g_{(2)},\\
\vartheta_H^+(z)g&=(-1)^{|z||g|}g_{(1)}\lra{g_{(2)}}{z},\quad z\in U(\fh),\ g\in\co(H).
\end{align*}

\begin{proposition}\label{prop:diag-rho}
For \(z\in\fh\) and \(g\in\co(H)\), we have
\begin{align*}
\rho_1(z)\,\Gamma^{-1}(1\otimes g\otimes 1)=&\Gamma^{-1}(1\otimes\vartheta_H^-(z)g\otimes 1),\\
\rho_2(z)\,\Gamma^{-1}(1\otimes g\otimes 1)=&\Gamma^{-1}(1\otimes\vartheta_H^+(z)g\otimes 1).
\end{align*}
\end{proposition}

\begin{proof}
We prove the first identity; the second is analogous. Set \(f=\Gamma^{-1}(1\otimes g\otimes 1)\). By \eqref{eq:dh1},
\[
(\id\otimes\,\Gamma)\,\delta_H^-(f)
=g_{(1)}\otimes 1\otimes g_{(2)}\otimes 1.
\]
Therefore,
\begin{align*}
\Gamma(\rho_1(z)f)&=(-1)^{|z|}((\lra{-}{S^{-1}(z)}\otimes\Gamma)\circ\delta_H^-)(f)\\
&=1\otimes (-1)^{|z|}\lra{g_{(1)}}{S^{-1}(z)}g_{(2)}\otimes 1\\
&=1\otimes \vartheta_H^-(z)g\otimes 1.
\end{align*}
Applying \(\Gamma^{-1}\) to both sides of this equality yields the first identity.
\end{proof}

\section{Highest weight vectors}\label{sec:singular}

\subsection{Singular vectors in \texorpdfstring{\(\co(\Omega)\)}{O(Omega)}}\label{subsec:sing}

We now use the Gauss decomposition to describe the singular vectors for the regular representations on \(\co(\Omega)\). We define the \((\fn^-\oplus\fn^+)\)-singular subspace of \(\co(\Omega)\) by
\[
\Sing(\co(\Omega)):=\{f\in\co(\Omega)\mid \rho_1(x)f=\rho_2(y)f=0,\forall x\in\fn^-,y\in\fn^+\}.
\]
Let \(\Gamma\) denote the Gauss decomposition isomorphism in Theorem~\ref{thm:gauss}.

\begin{theorem}\label{thm:singular-alpha}
The \((\fn^-\oplus\fn^+)\)-singular subspace of \(\co(\Omega)\) is given by
\[
\Sing(\co(\Omega))=\Gamma^{-1}(1\otimes \co(H)\otimes 1).
\]
\end{theorem}

For the proof of Theorem~\ref{thm:singular-alpha}, the key point is the following two observations on the regular representations of the nilpotent subalgebras.

\begin{lemma}\label{lem:sing-down}
For \(x\in U(\fn^-)\) and \(f\in\co(U^-)\), define
\[
\vartheta^-(x)f=(-1)^{|x|}\lra{f_{(1)}}{S^{-1}(x)}f_{(2)}.
\]
Then
\[
\{f\in \co(U^-)\mid \vartheta^-(x)f=0,\forall x\in \fn^-\}=\bc.
\]
\end{lemma}

\begin{proof}
Clearly every scalar is annihilated by \(\fn^-\). Conversely, let \(f\in\co(U^-)\) be annihilated by \(\fn^-\). Since \(U(\fn^-)\) is generated by \(\fn^-\) and \(1\), it follows that
\[
\vartheta^-(u)f=\epsilon(u)f
\]
for all \(u\in U(\fn^-)\). Applying this identity to \(u=S(v)\), for \(v\in U(\fn^-)\), gives
\[
(-1)^{|v|}(\lra{-}{v}\otimes\id)(\Delta(f)-1\otimes f)=0,
\]
since the antipode preserves parity and the counit. By linearity and the nondegeneracy of the canonical pairing between \(\co(U^-)\) and \(U(\fn^-)\), we have \(\Delta(f)=1\otimes f\). Applying \(\id\otimes\,\epsilon\), we obtain \(f=\epsilon(f)\in\bc\).
\end{proof}

\begin{lemma}\label{lem:sing-up}
For \(y\in U(\fn^+)\) and \(f\in\co(U^+)\), define
\[
\vartheta^+(y)f=(-1)^{|y||f|}f_{(1)}\lra{f_{(2)}}{y}.
\]
Then
\[
\{f\in \co(U^+)\mid \vartheta^+(y)f=0,\forall y\in \fn^+\}=\bc.
\]
\end{lemma}

\begin{proof}
The proof is analogous to that of Lemma~\ref{lem:sing-down}.
\end{proof}

\begin{proof}[Proof of Theorem~\ref{thm:singular-alpha}]
To treat \(\rho_1|_{\fn^-}\), consider the commutative diagram
\[
\begin{tikzcd}[column sep=large]
U^-\times U^-\times H\times U^+
\arrow[r,"\mu_{U^-}\times\id\times\id"]
\arrow[d,"\id\times\gamma"'] &
U^-\times H\times U^+ \arrow[d,"\gamma"] \\
U^-\times\Omega \arrow[r,"\mu_\Omega^-"']& \Omega.
\end{tikzcd}
\]
It gives
\[
(\id\otimes\,\Gamma)\circ\delta_\Omega^-=(\Delta\otimes\id\otimes\id)\circ\Gamma.
\]
Together with Proposition~\ref{prop:eta-rho}, this implies that, for \(f\in\co(\Omega)\) and \(x\in U(\fn^-)\),
\begin{align*}
\Gamma(\rho_1(x)f)&=(-1)^{|x|}
((\lra{-}{S^{-1}(x)}\otimes\Gamma)\circ\delta_\Omega^-)(f)\\
&=(-1)^{|x|}
((\lra{-}{S^{-1}(x)}\otimes\id\otimes\id\otimes\id)
\circ(\Delta\otimes\id\otimes\id)\circ\Gamma)(f)\\
&=((\vartheta^-(x)\otimes\id\otimes\id)\circ\Gamma)(f).
\end{align*}
Combining this with Lemma~\ref{lem:sing-down}, we have
\[
\{f\in\co(\Omega)\mid \rho_1(x)f=0,\forall x\in\fn^-\}=\Gamma^{-1}(1\otimes\co(H)\otimes \co(U^+)).
\]

Similarly, for \(\rho_2|_{\fn^+}\), the corresponding diagram gives
\[
(\Gamma\otimes\id)\circ\delta_\Omega^+=(\id\otimes\id\otimes\Delta)\circ\Gamma.
\]
Hence, for \(f\in\co(\Omega)\) and \(y\in U(\fn^+)\),
\[
\Gamma(\rho_2(y)f)=((\id\otimes\id\otimes\,\vartheta^+(y))\circ\Gamma)(f).
\]
Combining this with Lemma~\ref{lem:sing-up}, we have
\[
\{f\in\co(\Omega)\mid \rho_2(y)f=0,\ \forall y\in\fn^+\}=\Gamma^{-1}(\co(U^-)\otimes\co(H)\otimes 1).
\]
The theorem follows by combining these two results.
\end{proof}

\subsection{Highest weight vectors in Howe duality}\label{sec:sing}

We realize the Howe duality on \(\co(\cx(n))\subset\co(Q(n))\), thereby identifying its highest weight vectors with singular vectors for the regular representations on \(\co(\cx(n))\).

The multiplication map
\[
Q(n)\times \cx(n)\times Q(n)\to \cx(n)
\]
implies that the regular actions \(\rho_1,\rho_2\) of \(\fq(n)\) restrict to actions on \(\co(\cx(n))\). 

We identify both \(\co(\cx(n))\) and \(S\bigl(\bc^{n^2|n^2}\bigr)\) with the subalgebra
\[
\bc[x_{ij},y_{ij}\mid 1\leq i,j\leq n]\subset\cpinf.
\]
Let \(\sigma\) be the automorphism of \(\fq(n)\) defined by
\[
\sigma(E_{ij})=-E_{ji},\quad\sigma(F_{ij})=\sqrt{-1}F_{ji}
\quad (1\leq i,j\leq n).
\]

\begin{proposition}\label{prop:iso}
For \(x,y\in\fq(n)\) and \(f\in\co(\cx(n))\), we have
\[
\rho_1(x)f=\varrho_1(\sigma(x))f,\quad
\rho_2(y)f=\varrho_2(y)f.
\]
\end{proposition}

\begin{proof}
Under the identification above, both sides are superderivations of the polynomial superalgebra \(\bc[x_{ij},y_{ij}\mid 1\leq i,j\leq n]\). Hence the identities reduce to a direct computation on the generators \(x_{ij},y_{ij}\).
\end{proof}

Combining Proposition~\ref{prop:iso} with the description of the singular vectors in \(\co(\Omega)\), we obtain the following consequence.

\begin{corollary}\label{cor:sing}
The subspace of highest weight vectors in \(S(\bc^{n^2|n^2})\) is given by
\begin{align*}
&\left\{f\in S\bigl(\bc^{n^2|n^2}\bigr)\;\middle|\; \varrho_1(x)f=\varrho_2(y)f=0,\forall x,y\in\fn^+\right\}\\
=&\Sing(\co(\Omega))\cap\co(\cx(n))=\Gamma^{-1}(1\otimes\co(H)\otimes 1)\cap\co(\cx(n)).
\end{align*}
\end{corollary}

For \(r\in\bn\), set \(d_r:=\det(x_{ij})_{1\leq i,j\leq r}\). Let \(\cd\subset\cpinf\) be the multiplicative subset generated by \(d_r\) for all \(r\in\bn\), and consider the localization \(\cd^{-1}\cpinf\). By Corollary~\ref{cor:der-localization}, the \(\fq(\infty)\)-actions \(\varrho_1,\varrho_2\) on \(\cpinf\) extend to \(\fq(\infty)\)-actions on \(\cd^{-1}\cpinf\), which we still denote by \(\varrho_1,\varrho_2\).

For each \(n\), the coordinate superalgebra of the big cell \(\Omega_n\),
\[
\co(\Omega_n)=\bc[x_{ij},y_{ij}\mid 1\leq i,j\leq n]_{d_1\cdots d_n},
\]
naturally embeds into \(\cd^{-1}\cpinf\). By Proposition~\ref{prop:iso} and the uniqueness of superderivation extensions, the regular \(\fq(n)\)-actions \(\rho_1,\rho_2\) on \(\co(\Omega_n)\) coincide with the restrictions of the extended \(\fq(\infty)\)-actions. More precisely, for \(x,y\in\fq(n)\),
\begin{equation}\label{eq:iso}
\rho_1(x)=\varrho_1(\sigma(x))|_{\co(\Omega_n)},\quad
\rho_2(y)=\varrho_2(y)|_{\co(\Omega_n)}.    
\end{equation}

For \(n\in\bn\) and \(1\leq k\leq n\), let
\[
\Gamma_n:\co(\Omega_n)\overset{\sim}{\to}\co(U_n^-)\otimes\co(H_n)\otimes\co(U_n^+)
\]
be the Gauss decomposition isomorphism for \(Q(n)\), and set
\[
a_k^{(n)}:=\Gamma_n^{-1}(1\otimes x_{kk}\otimes 1),\quad
b_k^{(n)}:=\Gamma_n^{-1}(1\otimes y_{kk}\otimes 1)
\]
as elements of \(\co(\Omega_n)\).

For \(m\leq n\), take the natural embeddings of superschemes
\[
\Omega_m\hookrightarrow\Omega_n,\quad
U_m^\pm\hookrightarrow U_n^\pm,\quad
H_m\hookrightarrow H_n
\]
compatible with the multiplication maps in the Gauss decompositions. The corresponding pullback \(\co(\Omega_n)\to\co(\Omega_m)\) sends \(a_k^{(n)},b_k^{(n)}\) to \(a_k^{(m)},b_k^{(m)}\) for \(1\leq k\leq m\). Hence, for each \(k\), the elements \(a_k^{(n)},b_k^{(n)}\) with \(n\geq k\) have common images in \(\cd^{-1}\cpinf\).

\begin{definition}\label{def:ab}
For \(k\in\bn\), we define \(a_k,b_k\in\cd^{-1}\cpinf\) to be the common images of \(a_k^{(n)},b_k^{(n)}\) for \(n\geq k\), respectively.
\end{definition}

We now use the elements \(a_k,b_k\in\cd^{-1}\cpinf\) defined above to describe the highest weight vectors in the Howe duality.

\begin{theorem}\label{thm:highest}
Let \(\lambda=(\lambda_1,\ldots,\lambda_\ell)\) be a strict partition. Then
\[
\left\{\prod_{k=1}^\ell a_k^{\lambda_k-\delta_k}b_k^{\delta_k}\;\middle|\; \delta_k\in\{0,1\}\ (1\leq k\leq\ell)\right\}
\]
forms a basis of \(H_\lambda\).
\end{theorem}

\begin{proof}
Choose \(n\geq\ell\). Theorem~\ref{thm:singular-alpha}, applied to \(\Omega_n\), shows that the \(\fn^-\oplus\fn^+\)-singular subspace of \(\co(\Omega_n)\) of weight \((-\lambda,\lambda)\) has the set in the statement as a basis. This set has \(2^\ell\) elements, which equals \(\dim H_\lambda\). The result follows from Corollary~\ref{cor:sing}.
\end{proof}

The following formula gives an explicit expression for \(a_k,b_k\in\cd^{-1}\cpinf\).

\begin{proposition}\label{prop:formula}
For \(i,j,k\in\bn\), define matrices \(C_k,Z_{ij}\in M_2(\cd^{-1}\cpinf)\) by
\[
C_k:=
\begin{pmatrix}
a_k&b_k\\
b_k&a_k
\end{pmatrix},
\quad
Z_{ij}:=
\begin{pmatrix}
x_{ij}&y_{ij}\\
y_{ij}&x_{ij}
\end{pmatrix}.
\]
Then, for \(k\in\bn\), the elements \(a_k,b_k\in\cd^{-1}\cpinf\) are determined by
\begin{equation}\label{eq:akbk}
C_k=Z_{kk}-(Z_{k1},\ldots,Z_{k,k-1})(Z_{ij})_{1\leq i,j\leq k-1}^{-1}(Z_{1k},\ldots,Z_{k-1,k})^t.
\end{equation}
\end{proposition}

\begin{proof}
Choose \(n\geq k\) and work in \(\co(\Omega_n)\). By the definition of the Gauss decomposition isomorphism \(\Gamma_n\), the block matrix \((Z_{ij})_{1\leq i,j\leq n}\in M_n(M_2(\co(\Omega_n)))\) admits the factorization
\[
(Z_{ij})_{1\leq i,j\leq n}=L\diag(C_1^,\ldots,C_n)R,
\]
where \(L\) and \(R\) are lower and upper unipotent block matrices, respectively.

The Schur complement of the leading \((k-1)\times(k-1)\) block in the leading \(k\times k\) block matrix is invariant under left multiplication by lower unipotent matrices and right multiplication by upper unipotent matrices; see \cite[Theorem~1.3]{MR2160825}. Therefore, this factorization gives
\[
C_k=Z_{kk}-
(Z_{k1},\ldots,Z_{k,k-1})
(Z_{ij})_{1\leq i,j\leq k-1}^{-1}
(Z_{1k},\ldots,Z_{k-1,k})^t.
\]
Taking the image in \(\cd^{-1}\cpinf\) gives the desired identity.
\end{proof}

\begin{remark}
Formula~\eqref{eq:akbk} may also be used as a direct definition of elements \(a_k,b_k\in\cd^{-1}\cpinf\) for \(k\in\bn\). With this definition, Proposition~\ref{prop:formula} states that these elements coincide with the diagonal factors obtained from the Gauss decompositions of queer supergroups.
\end{remark}

\begin{corollary}\label{cor:polynomiality}
Let \(a_k,b_k\in\cd^{-1}\cpinf\) be defined by Formula~\eqref{eq:akbk}. If \(\lambda=(\lambda_1,\ldots,\lambda_\ell)\) is a strict partition, then, for any \((\delta_1,\ldots,\delta_\ell)\in\{0,1\}^\ell\),
\[
\prod_{k=1}^{\ell}a_k^{\lambda_k-\delta_k}b_k^{\delta_k}\in\cpinf.
\]
\end{corollary}

\begin{proof}
By Theorem~\ref{thm:highest}, these elements form a basis of \(H_\lambda\). Since \(H_\lambda\) is a subspace of \(\cpinf\), the assertion follows.
\end{proof}

\begin{example}
As an example, we explicitly compute a basis of \(H_\lambda\) for \(n=2\) and \(\lambda=(2,1)\). The basis is
\[
a_1^2a_2,\quad a_1^2b_2,\quad a_1b_1a_2,\quad a_1b_1b_2.
\]
We first compute \(a_r,b_r\) for \(r=1,2\).

For \(r=1\), we have
\[
C_1=\begin{pmatrix}
x_{11}&y_{11}\\y_{11}&x_{11}
\end{pmatrix}.
\]
Hence, \(a_1=x_{11}\) and \(b_1=y_{11}\).

For \(r=2\), we have
\begin{align*}
C_2&=\begin{pmatrix}
x_{22}&y_{22}\\y_{22}&x_{22}
\end{pmatrix}-
\begin{pmatrix}
x_{21}&y_{21}\\y_{21}&x_{21}
\end{pmatrix}
\begin{pmatrix}
x_{11}&y_{11}\\y_{11}&x_{11}
\end{pmatrix}^{-1}
\begin{pmatrix}
x_{12}&y_{12}\\y_{12}&x_{12}
\end{pmatrix}
\\
&=\begin{pmatrix}
x_{22}&y_{22}\\y_{22}&x_{22}
\end{pmatrix}-
\begin{pmatrix}
x_{21}&y_{21}\\y_{21}&x_{21}
\end{pmatrix}
\begin{pmatrix}
x_{11}^{-1}&-x_{11}^{-2}y_{11}\\-x_{11}^{-2}y_{11}&x_{11}^{-1}
\end{pmatrix}
\begin{pmatrix}
x_{12}&y_{12}\\y_{12}&x_{12}
\end{pmatrix}.
\end{align*}
It follows that
\begin{align*}
a_2&=x_{22}-x_{21}x_{11}^{-1}x_{12}+x_{21}x_{11}^{-2}y_{11}y_{12}+y_{21}x_{11}^{-2}y_{11}x_{12}-y_{21}x_{11}^{-1}y_{12}\\
&=x_{11}^{-2}(x_{11}^2x_{22}-x_{11}x_{12}x_{21}+x_{21}y_{11}y_{12}-x_{12}y_{11}y_{21}+x_{11}y_{12}y_{21}),
\end{align*}
and
\begin{align*}
b_2&=y_{22}-x_{21}x_{11}^{-1}y_{12}+x_{21}x_{11}^{-2}y_{11}x_{12}+y_{21}x_{11}^{-2}y_{11}y_{12}-y_{21}x_{11}^{-1}x_{12}\\
&=x_{11}^{-2}(x_{11}^2y_{22}-x_{11}x_{21}y_{12}+x_{12}x_{21}y_{11}+y_{11}y_{12}y_{21}-x_{11}x_{12}y_{21}).
\end{align*}

Therefore, we obtain the following four highest weight vectors:
\begin{align*}
a_1^2a_2&=x_{11}^2x_{22}-x_{11}x_{12}x_{21}+x_{21}y_{11}y_{12}-x_{12}y_{11}y_{21}+x_{11}y_{12}y_{21},\\
a_1^2b_2&=x_{11}^2y_{22}-x_{11}x_{21}y_{12}+x_{12}x_{21}y_{11}+y_{11}y_{12}y_{21}-x_{11}x_{12}y_{21},\\
a_1b_1a_2&=x_{11}x_{22}y_{11}-x_{12}x_{21}y_{11}+y_{11}y_{12}y_{21},\\
a_1b_1b_2&=x_{11}y_{11}y_{22}-x_{21}y_{11}y_{12}-x_{12}y_{11}y_{21}.
\end{align*}
\end{example}

\subsection{Properties of highest weight vectors}

Among the basis vectors in Theorem~\ref{thm:highest}, the one with only \(a_k\)-factors has a distinguished role. We consider \(H_\lambda\) as an irreducible \(\Cl_{2\ell}\)-module via the module structure defined in Proposition~\ref{prop:cli}.

\begin{proposition}\label{prop:hpure}
The vector
\[
v_\lambda:=\prod_{k=1}^\ell a_k^{\lambda_k}
\]
is a pure spinor in \(H_\lambda\) with respect to \(c_1^-,\ldots,c_\ell^-\).
\end{proposition}

\begin{proof}
Choose \(n\geq\ell\) and work in \(\co(\Omega_n)\). We verify directly that \(c_i^-v_\lambda=0\) for all \(1\leq i\leq\ell\), that is,
\[
(\varrho_1(F_{ii})+\sqrt{-1}\varrho_2(F_{ii}))v_\lambda=0\quad(1\leq i\leq \ell).
\]

We begin by determining the actions on each factor \(a_k\), \(1\leq k\leq\ell\). Since \(\Gamma_n(a_k)=1\otimes x_{kk}\otimes 1\), for \(1\leq i\leq\ell\), Proposition~\ref{prop:diag-rho} gives
\begin{align*}
\rho_1(F_{ii})a_k&=\Gamma_n^{-1}(1\otimes\vartheta_H^-(F_{ii})x_{kk}\otimes 1),\\
\rho_2(F_{ii})a_k&=\Gamma_n^{-1}(1\otimes\vartheta_H^+(F_{ii})x_{kk}\otimes 1).
\end{align*}
On \(\co(H)\), we have
\[
\vartheta_H^-(F_{ii})x_{kk}
=\vartheta_H^+(F_{ii})x_{kk}
=\sqrt{-1}\delta_{ik}y_{kk}.
\]
Hence,
\[
\rho_1(F_{ii})a_k=\rho_2(F_{ii})a_k=\sqrt{-1}\delta_{ik}b_k.
\]
By \eqref{eq:iso}, we have \(\rho_1(F_{ii})=\sqrt{-1}\varrho_1(F_{ii})\), \(\rho_2(F_{ii})=\varrho_2(F_{ii})\) on \(\co(\Omega_n)\). Hence
\[
\varrho_1(F_{ii})a_k=\delta_{ik}b_k,\quad\varrho_2(F_{ii})a_k=\sqrt{-1}\delta_{ik}b_k.
\]

Thus, for each \(1\leq i\leq\ell\), the operator \(\varrho_1(F_{ii})+\sqrt{-1}\varrho_2(F_{ii})\) annihilates every factor \(a_k\). Since this operator is a superderivation, the super Leibniz rule gives
\[
(\varrho_1(F_{ii})+\sqrt{-1}\varrho_2(F_{ii}))\biggl(\prod_{k=1}^\ell a_k^{\lambda_k}\biggr)=0
\quad (1\leq i\leq\ell).\qedhere
\]
\end{proof}

\begin{remark}
The basis \eqref{eq:basis} generated from the pure spinor \(v_\lambda\) by \(c_1^+,\ldots,c_\ell^+\) is precisely the basis in Theorem~\ref{thm:highest}. Indeed, by Proposition~\ref{prop:cli} and the computation in the proof above, we have
\[
c_i^+a_k
=\frac{1}{2}\bigl(\varrho_1(F_{ii})-\sqrt{-1}\varrho_2(F_{ii})\bigr)a_k
=\delta_{ik}b_k\quad(1\leq i,k\leq \ell).
\]
Since \(c_i^+\) acts by a superderivation, the super Leibniz rule gives that, for
\[
I=\{i_1<\cdots<i_r\}\subseteq[\ell],
\]
the vector \(c_{i_1}^+\cdots c_{i_r}^+v_\lambda\) is a nonzero scalar multiple of
\[
\prod_{k=1}^{\ell}a_k^{\lambda_k-\delta_k}b_k^{\delta_k},\quad \delta_k=
\begin{cases}
1,& k\in I,\\
0,& k\notin I.
\end{cases}
\]
\end{remark}

It is well known that the highest weight vectors in the \((\fgl(m),\fgl(n))\)-Howe duality are products of leading minors, and hence invariant under the transpose involution \(x_{ij}\mapsto x_{ji}\). In the queer setting, the following proposition establishes an analogous invariance property for \(a_k\) and \(b_k\).

Let \(\omega\) be the anti-involution of \(\cd^{-1}\cpinf\) determined by
\[
\omega(x_{ij})=x_{ji},\quad \omega(y_{ij})=y_{ji}\quad (i,j\in\bn).
\]

\begin{proposition}\label{prop:omega}
For \(k\in\bn\), one has
\[
\omega(a_k)=a_k,\quad\omega(b_k)=b_k.
\]
\end{proposition}

\begin{proof}
Extending \(\omega\) entrywise to matrices, for any two matrices \(M_1\) and \(M_2\) with compatible sizes, we have
\[
\omega(M_1M_2)^t=\omega(M_2)^t\omega(M_1)^t.
\]
By Proposition~\ref{prop:formula}, we have
\begin{align*}
&\omega(C_k)=\omega(Z_{kk})-\omega((Z_{k1},\ldots,Z_{k,k-1})
(Z_{ij})_{1\leq i,j\leq k-1}^{-1}
(Z_{1k},\ldots,Z_{k-1,k})^t)\\
=&Z_{kk}-(\omega((Z_{1k},\ldots,Z_{k-1,k}))\omega((Z_{ji})_{1\leq i,j\leq k-1}^{-1})\omega((Z_{k1},\ldots,Z_{k,k-1})^t))^t\\
=&Z_{kk}-((Z_{k1},\ldots,Z_{k,k-1})
(Z_{ij})_{1\leq i,j\leq k-1}^{-1}
(Z_{1k},\ldots,Z_{k-1,k})^t)^t=C_k.\qedhere
\end{align*}
\end{proof}

\section{Application to the general linear Lie superalgebra}\label{sec:gl-case}

The method developed in the previous sections applies to the case \(p=m\) of \((\fgl(p|q),\fgl(m|n))\)-Howe duality. Our goal is to recover, in this case, the explicit highest weight vector formulas of Cheng and Wang; see \cite[Theorems~4.3 and~4.4]{MR1847665}.

\subsection{Howe duality for the general linear Lie superalgebra}

We recall the Howe duality for the general linear Lie superalgebras.

\begin{theorem}[{see \cite[Theorem~3.2]{MR1847665}}]\label{thm:howegl}
As a \(\fgl(p|q)\oplus\fgl(m|n)\)-module, \(S(\bc^{p|q}\otimes \bc^{m|n})\) admits the multiplicity-free decomposition
\[
S(\bc^{p|q}\otimes \bc^{m|n})\cong\bigoplus_{\substack{\lambda_{p+1}\leq q,\\\lambda_{m+1}\leq n}}L_{p|q}(\lambda)\otimes L_{m|n}(\lambda),
\]
where \(L_{r|s}(\lambda)\) denotes the irreducible \(\fgl(r|s)\)-module of highest weight \((\lambda_1,\ldots,\lambda_r,\langle\lambda'_1-r\rangle,\ldots,\langle\lambda'_s-r\rangle)\), where \(\langle k\rangle:=\max(k,0)\).
\end{theorem}

We denote the standard basis of \(\fgl(r|s)\) by
\[
\{E_{ij},E_{\bar k\bar l},F_{i\bar l},F_{\bar k j}
\mid 1\leq i,j\leq r,\ 1\leq k,l\leq s\}.
\]
We now choose explicit coordinates on \(S(\bc^{p|q}\otimes\bc^{m|n})\), identifying it with a polynomial superalgebra on which the two general linear Lie superalgebras act by differential operators.

Let \(\{e_1,\ldots,e_p,f_1,\ldots,f_q\}\) and \(\{u_1,\ldots,u_m,v_1,\ldots,v_n\}\) be standard bases of \(\bc^{p|q}\) and \(\bc^{m|n}\), respectively, where \(e_i,u_j\) are even and \(f_k,v_l\) are odd. Then we identify \(S(\bc^{p|q}\otimes\bc^{m|n})\) with the polynomial superalgebra generated by
\begin{align*}
x_{ij}:=e_i\otimes u_j,\quad \eta_{il}:=e_i\otimes v_l,\quad\xi_{kj}:=f_k\otimes u_j,\quad y_{kl}&:=f_k\otimes v_l,
\end{align*}
where \(x_{ij},y_{kl}\) are even and \(\eta_{il},\xi_{kj}\) are odd, with
\[
1\leq i\leq p,\quad 1\leq k\leq q,\quad
1\leq j\leq m,\quad 1\leq l\leq n.
\]
Under this identification, the \(\fgl(p|q)\)-action is given by
\begin{align*}
\varrho_1(E_{ii'})=&\sum_{j=1}^m x_{ij}\frac{\partial}{\partial x_{i'j}}
+\sum_{l=1}^n \eta_{il}\frac{\partial}{\partial \eta_{i'l}},&
\varrho_1(E_{\bar k\bar k'})=&\sum_{j=1}^m \xi_{kj}\frac{\partial}{\partial \xi_{k'j}}
+\sum_{l=1}^n y_{kl}\frac{\partial}{\partial y_{k'l}},\\
\varrho_1(F_{i\bar k})=&\sum_{j=1}^m x_{ij}\frac{\partial}{\partial \xi_{kj}}
+\sum_{l=1}^n \eta_{il}\frac{\partial}{\partial y_{kl}},&
\varrho_1(F_{\bar k i})=&\sum_{j=1}^m \xi_{kj}\frac{\partial}{\partial x_{ij}}
+\sum_{l=1}^n y_{kl}\frac{\partial}{\partial \eta_{il}},
\end{align*}
where \(1\leq i,i'\leq p,\ 1\leq k,k'\leq q\), and the \(\fgl(m|n)\)-action is given by
\begin{align*}
\varrho_2(E_{jj'})=&\sum_{i=1}^p x_{ij}\frac{\partial}{\partial x_{ij'}}
+\sum_{k=1}^q \xi_{kj}\frac{\partial}{\partial \xi_{kj'}},&
\varrho_2(E_{\bar l\bar l'})=&\sum_{i=1}^p \eta_{il}\frac{\partial}{\partial \eta_{il'}}
+\sum_{k=1}^q y_{kl}\frac{\partial}{\partial y_{kl'}},\\
\varrho_2(F_{j\bar l})=&\sum_{i=1}^p x_{ij}\frac{\partial}{\partial \eta_{il}}
-\sum_{k=1}^q \xi_{kj}\frac{\partial}{\partial y_{kl}},&
\varrho_2(F_{\bar l j})=&\sum_{i=1}^p \eta_{il}\frac{\partial}{\partial x_{ij}}
-\sum_{k=1}^q y_{kl}\frac{\partial}{\partial \xi_{kj}},
\end{align*}
where \(1\leq j,j'\leq m,\ 1\leq l,l'\leq n\).

\subsection{Regular representations on coordinate superalgebras}

From now on, we restrict to the case \(p=m\). Consider the directed system of polynomial superalgebras \(S(\bc^{m|q}\otimes\bc^{m|n})\), indexed by \(n,q\in\bn\), under the natural inclusions. Its direct limit is the polynomial superalgebra
\[
\cpinf^{(m)}:=\bc[x_{ij},\eta_{il},\xi_{kj},y_{kl}\mid
1\leq i,j\leq m,\ k,l\in\bn],
\]
where \(x_{ij},y_{kl}\) are even and \(\eta_{il},\xi_{kj}\) are odd. Under the canonical maps to this direct limit, an argument analogous to that in Remark~\ref{rmk:reduction-square} identifies, for \(q,n\geq\lambda_{m+1}\), the highest weight vector spaces corresponding to the \(\lambda\)-summands in the \((\fgl(m|q),\fgl(m|n))\)-Howe dualities. We denote the resulting one-dimensional subspace of \(\cpinf^{(m)}\) by \(H_\lambda^{(m)}\). For the purpose of characterizing these vectors, it suffices to consider the \((\fgl(m|n),\fgl(m|n))\)-case.

For \(1\leq r\leq m\) and \(s\in\bn\), set
\[
d_r^X:=\det(x_{ij})_{1\leq i,j\leq r},\quad
d_s^Y:=\det(y_{kl})_{1\leq k,l\leq s}.
\]
Let \(\cd_m\subset\cpinf^{(m)}\) be the multiplicative subset generated by \(d_r^X\ (1\leq r\leq m)\) and \(d_s^Y\ (s\in\bn)\). The natural map embeds \(\cpinf^{(m)}\) into the localization \(\cd_m^{-1}\cpinf^{(m)}\).

In parallel with the definitions in the preceding sections, one defines the affine superspace \(M(m|n)\), the general linear supergroup \(\GL(m|n)\), the unipotent subgroups \(U^\pm\), the diagonal subgroup \(H\), and the big cell \(\Omega\) associated with the Gauss decomposition. 

Let \(I_e:=\ker\epsilon\) be the augmentation ideal of \(\co(\GL(m|n))\). The residue classes of the elements in
\[
\{x_{ij}-\delta_{ij},\, y_{kl}-\delta_{kl},\, \eta_{il},\, \xi_{kj}
\mid 1\leq i,j\leq m,\ 1\leq k,l\leq n\}
\]
constitute a basis of \(I_e/I_e^2\). Let
\[
\{E'_{ij},\, E'_{\bar k\bar l},\, F'_{i\bar l},\, F'_{\bar k j}
\mid 1\leq i,j\leq m,\ 1\leq k,l\leq n\}
\]
be the corresponding dual basis, viewed as a basis of \(\Lie(\GL(m|n))\). Then we identify \(\fgl(m|n)\) with \(\Lie(\GL(m|n))\) by
\begin{align*}
E_{ij}&\mapsto E'_{ij},&
E_{\bar k\bar l}&\mapsto E'_{\bar k\bar l},\\
F_{i\bar l}&\mapsto -F'_{i\bar l},&
F_{\bar k j}&\mapsto F'_{\bar k j}
\quad (1\leq i,j\leq m,\ 1\leq k,l\leq n).
\end{align*}

The regular representations \(\rho_1,\rho_2\) and singular vectors are obtained in the same way as in Subsections~\ref{sec:regular-big-cell} and~\ref{subsec:sing}.

Under the identification \(\co(M(m|n))\cong S(\bc^{m|n}\otimes\bc^{m|n})\), the regular representations on \(\co(M(m|n))\) correspond to the Howe duality. Let \(\sigma\) be the automorphism of \(\fgl(m|n)\) defined by
\begin{align*}
\sigma(E_{ij})&=-E_{ji},&
\sigma(E_{\bar k\bar l})&=-E_{\bar l\bar k},\\
\sigma(F_{i\bar l})&=-F_{\bar l i},&
\sigma(F_{\bar k j})&=F_{j\bar k}
\quad (1\leq i,j\leq m,\ 1\leq k,l\leq n).
\end{align*}

\begin{proposition}\label{prop:gliso}
For \(x,y\in\fgl(m|n)\) and \(f\in\co(M(m|n))\), we have
\[
\rho_1(x)f=\varrho_1(\sigma(x))f,\quad\rho_2(y)f=\varrho_2(y)f.
\]
\end{proposition}

For each \(n\), let \(\Omega_{m|n}\) denote the big cell of \(\GL(m|n)\). Set
\[
D_{m|n}:=d_1^X\cdots d_m^X d_1^Y\cdots d_n^Y.
\]
Then
\[
\co(\Omega_{m|n})
=\bc[x_{ij},y_{kl},\eta_{il},\xi_{kj}
\mid 1\leq i,j\leq m,\ 1\leq k,l\leq n]_{D_{m|n}},
\]
which naturally embeds into \(\cd_m^{-1}\cpinf^{(m)}\). These embeddings are compatible with the natural inclusions in \(n\), as well as with the corresponding inclusions of the Lie superalgebras and their regular actions.

\subsection{Highest weight vectors in Howe duality}

Let
\[
\Gamma:\co(\Omega)\overset{\sim}{\to}\co(U^-)\otimes\co(H)\otimes\co(U^+)
\]
be the Gauss decomposition isomorphism. By the same argument as in Theorem~\ref{thm:singular-alpha}, the \(\fn^-\oplus\fn^+\)-singular subspace is given by
\[
\Sing(\co(\Omega))=\Gamma^{-1}(1\otimes\co(H)\otimes 1).
\]
Combining it with Proposition~\ref{prop:gliso}, the space of highest weight vectors in the \(\fgl(m|n)\oplus\fgl(m|n)\)-module \(S(\bc^{m|n}\otimes\bc^{m|n})\) is given by
\begin{align*}
&\left\{
f\in S(\bc^{m|n}\otimes\bc^{m|n})\;\middle|\;
\varrho_1(x)f=\varrho_2(y)f=0,\forall x,y\in\fn^+
\right\}\\
=&\Sing(\co(\Omega))\cap\co(M(m|n))=\Gamma^{-1}(1\otimes\co(H)\otimes 1)\cap\co(M(m|n)).
\end{align*}

This description allows us to write down explicit highest weight vectors for the corresponding summands in the Howe duality.

Fix \(m\). For \(r\in\bn\), \(1\leq i\leq m\), and \(1\leq j\leq r\), set
\[
\alpha_i^{(m|r)}:=\Gamma_{m|r}^{-1}(1\otimes x_{ii}\otimes 1),\quad
\beta_j^{(m|r)}:=\Gamma_{m|r}^{-1}(1\otimes y_{jj}\otimes 1).
\]
The same compatibility argument used before Definition~\ref{def:ab} shows that the following common images are well-defined.
\begin{definition}\label{def:glab-global}
For fixed \(m\), we define
\[
\alpha_i\ (1\leq i\leq m),\quad \beta_j^{(m)}\ (j\in\bn)
\]
to be the common images of \(\alpha_i^{(m|r)}\) for \(r\in\bn\) and of \(\beta_j^{(m|r)}\) for \(r\geq j\), respectively, in \(\cd_m^{-1}\cpinf^{(m)}\).
\end{definition}

\begin{theorem}\label{thm:glhighest}
For fixed \(m\), let \(\lambda\) be a partition, and define \(w_\lambda\in\cd_m^{-1}\cpinf^{(m)}\) by
\[
w_\lambda:=\prod_{i=1}^m \alpha_i^{\lambda_i}\cdot\prod_{j=1}^{\lambda_{m+1}} \bigl(\beta_j^{(m)}\bigr)^{\langle\lambda'_j-m\rangle}.
\]
Then \(H_\lambda^{(m)}=\bc w_\lambda\). In particular, for \(q,n\geq\lambda_{m+1}\), under the identification, the element \(w_\lambda\) is a highest weight vector for the \(L_{m|q}(\lambda)\otimes L_{m|n}(\lambda)\) summand.
\end{theorem}

Analogously, using Schur complements, we obtain the following explicit formulas for \(\alpha_i\) and \(\beta_j^{(m)}\). Define matrices over \(\cd_m^{-1}\cpinf^{(m)}\) by
\begin{align*}
\bfX&=(x_{ij})_{1\leq i,j\leq m},
&\bfH&=(\eta_{il})_{1\leq i\leq m,\ l\in\bn},\\
\bfXi&=(\xi_{kj})_{k\in\bn,\ 1\leq j\leq m},
&\bfY&=(y_{kl})_{k,l\in\bn}.
\end{align*}

\begin{proposition}\label{prop:glab}
The elements \(\alpha_i\) \((1\leq i\leq m)\) and
\(\beta_j^{(m)}\) \((j\in\bn)\) are determined by the following formulas.
\begin{align*}
&\prod_{i=1}^r \alpha_i=\det\bfX_{[r],[r]}&&(1\leq r\leq m),\\
&\prod_{j=1}^s \beta_j^{(m)}=\det\,(\bfY-\bfXi \bfX^{-1}\bfH)_{[s],[s]}&&(s\in\bn).
\end{align*}
\end{proposition}

\subsection{Comparison with Cheng--Wang's formula}

In this subsection, we show that the highest weight vectors obtained above recover the \(p=m\) case of Cheng--Wang's explicit formulas in \cite{MR1847665}. Our notation for the Howe duality corresponds to the notation in \cite{MR1847665} as follows:
\begin{align*}
x_{ij}\leftrightarrow x^i_j,\quad y_{kl}\leftrightarrow y^k_l,\quad
\eta_{il}\leftrightarrow\eta^i_l,\quad\xi_{kj}\leftrightarrow\xi^k_j
\quad (1\leq i,j\leq m,\ 1\leq k,l\leq n).
\end{align*}

It suffices to verify that, for the rectangular partition \(\lambda=(r^{m+1})\) with \(r\leq n\), the element \(\Gamma_r\) defined in Theorem~4.3 of \cite{MR1847665} coincides with the element \(w_\lambda\) defined in Theorem~\ref{thm:glhighest}. We write \(\rdet\) and \(\cdet\) for the row determinant and the column determinant, respectively.

Cheng--Wang's explicit formula is
\[
\Gamma_r=\sum_D (-1)^{|D|(|D|-1)/2}\biggl(\prod_{k=1}^r \rdet X_k\biggr)\rdet Y_D,
\]
where the sum ranges over all possible marked \(m\times r\) diagrams \(D\). We refer the reader to \cite{MR1847665} for the definitions of marked diagrams, \(X_k\) and \(Y_D\).

By Theorem~\ref{thm:glhighest} and Proposition~\ref{prop:glab}, we have
\begin{align*}
w_\lambda&=(\alpha_1\cdots \alpha_m)^r(\beta_1^{(m)}\cdots \beta_r^{(m)})\\
&=(\det\bfX)^r\det\,(\bfY-\bfXi \bfX^{-1}\bfH)_{[r],[r]}.
\end{align*}
For \(0\leq i\leq m\) and \(j\in\bn\), let \(\bfX_{i\leftarrow j}\) denote \(\bfX\) if \(i=0\), and otherwise the matrix obtained from \(\bfX\) by replacing its \(i\)-th column with the \(j\)-th column of \(\bfH\), and define \(u_{ij}:=\cdet \bfX_{i\leftarrow j}\). For \(1\leq i\leq m\), write \(w_i\) for the \(i\)-th column of \(\bfXi_{[r],[m]}\). For \(j\in\bn\), write \(v_j\) for the \(j\)-th column of \(\bfY_{[r],\bn}\). Using the adjugate matrix, we write \(\bfX^{-1}=(\det\bfX)^{-1}\bfX^{\adj}\). Then
\begin{align*}
w_\lambda&=\det\,\bigl((\det\bfX)\bfY-\bfXi (\bfX^{\adj})\bfH\bigr)_{[r],[r]}\\
&=\cdet\biggl((\det\bfX)v_1-\sum_{i=1}^m w_i u_{i1},\ldots,(\det\bfX)v_r-\sum_{i=1}^m w_i u_{ir}\biggr).
\end{align*}
For \(f\in\Hom(\{1,\ldots,r\},\{0,\ldots,m\})\), define the \(r\times r\) matrix \(\bfY_f\) by declaring its \(k\)-th column to be
\[
\begin{cases}
v_k, & f(k)=0,\\
w_{f(k)}, & f(k)>0,
\end{cases}
\]
and
\[
|f|:=\#\{1\leq k\leq r\mid f(k)>0\}.
\]
By the multilinearity of the column determinant,
\[
w_\lambda=\sum_{f\in\Hom(\{1,\ldots,r\},\{0,\ldots,m\})}(-1)^{|f|(|f|-1)/2}\biggl(\prod_{k=1}^r u_{f(k),k}\biggr)\cdet \bfY_f.
\]

Next, we prove that this expression coincides with the definition of \(\Gamma_r\). For \(f\in\Hom(\{1,\ldots,r\},\{0,\ldots,m\})\), let \(D_f\) be the marked diagram whose marked boxes are
\[
\{(f(k),k)\mid 1\leq k\leq r,\ f(k)>0\}.
\]
The assignment \(f\mapsto D_f\) gives a bijection between \(\Hom(\{1,\ldots,r\},\{0,\ldots,m\})\) and the set of all possible marked \(m\times r\) diagrams. By construction, \(|f|=|D_f|\). Moreover, it follows directly from the definitions of \(X_k\) and \(Y_D\) in
\cite{MR1847665} that
\[
\bfX_{f(k)\leftarrow k}=(X_k)^t,\quad \bfY_f=(Y_{D_f})^t,
\]
where \(X_k\) is associated with \(D_f\). Therefore,
\begin{align*}
w_\lambda&=\sum_{f\in\Hom(\{1,\ldots,r\},\{0,\ldots,m\})}(-1)^{|f|(|f|-1)/2}\biggl(\prod_{k=1}^r\cdet\bfX_{f(k)\leftarrow k}\biggr)\cdet \bfY_f\\
&=\sum_D (-1)^{|D|(|D|-1)/2}\biggl(\prod_{k=1}^r\rdet X_k\biggr)\rdet Y_D=\Gamma_r.
\end{align*}
Thus, the two explicit expressions coincide.

\section*{Acknowledgment}

The fourth author expresses sincere gratitude to Prof. Jae-Hoon Kwon for addressing the question and to Prof. Shun-Jen Cheng for their valuable suggestions.
\bigskip

\textbf{Declaration on the Use of Generative AI:} The authors used OpenAI’s ChatGPT
to improve the language and presentation of the manuscript and as an additional tool to
verify the logical consistency of some proofs. All mathematical arguments, statements,
and references were independently verified by the authors, who take full responsibility for
the content of the paper.
\bigskip

\textbf{Data Availability Statement:} Not applicable
\bigskip

\textbf{Conflicts of Interest:} The authors have no conflicts of interest to declare that are relevant to this article.

\bibliographystyle{amsalpha}
\bibliography{refs}

\end{document}